\documentclass{amsart}
\usepackage{graphicx}
\usepackage[latin1]{inputenc}
\usepackage{amssymb}
\usepackage{amsmath}
\usepackage{mathtools}
\usepackage{esint}
\usepackage{hyperref}
\usepackage{multirow}
\usepackage{multicol}
\usepackage{graphicx}
\usepackage{mathrsfs}
\usepackage{nicefrac}

\newtheorem{theorem}{Theorem}[section]
\newtheorem*{theorem*}{Theorem}
\newtheorem{lemma}[theorem]{Lemma}
\newtheorem{proposition}[theorem]{Proposition}

\theoremstyle{remark}
\newtheorem{remark}[theorem]{Remark}

\theoremstyle{definition}

\numberwithin{equation}{section}
\numberwithin{table}{section}
\numberwithin{figure}{section}

\newcommand{\rn}{{\mathbb{R}^n}}

\newcommand{\pp}{\partial}
\newcommand{\w}{\omega}
\newcommand{\W}{\Omega}
\newcommand{\eps}{\varepsilon}
\newcommand{\llangle}{\left\langle}
\newcommand{\rrangle}{\right\rangle}
\usepackage{color}

\def\R{{\mathbb {R}}}
\def\N{{\mathbb {N}}}
\def\T{{\mathcal {T}}}
\def\V{{\mathbb {V}}}
\title[FE approximation for the fractional eigenvalue problem]{Finite element approximation for the fractional eigenvalue problem}

\author[J. P. Borthagaray]{Juan Pablo Borthagaray}

\address{J.P. Borthagaray \hfill\break\indent IMAS - CONICET and 
Departamento de Matem\'a\-tica, 
\hfill\break\indent
FCEyN - Universidad de Buenos Aires, 
\hfill\break\indent
Ciudad Universitaria, 
\hfill\break\indent
Pabell\'on I  (1428) Buenos Aires, Argentina.}
\email{jpbortha@dm.uba.ar}

\thanks{Research of the first author has been partially supported by CONICET under grant PIP 2014-2016 11220130100184CO}

\author[L. M. Del Pezzo]{Leandro M. Del Pezzo}
	\address{Leandro M. Del Pezzo \hfill\break\indent
		CONICET and UTDT \hfill\break\indent
		Departamento de Matem\'aticas y 
		Estad\'istica
		\hfill\break\indent Universidad Torcuato Di Tella
		\hfill\break\indent Av. Figueroa Alcorta 7350 (C1428BCW)
		\hfill\break\indent Buenos Aires, ARGENTINA. }
	\email{ldelpezzo@utdt.edu}
	\urladdr{http://cms.dm.uba.ar/Members/ldpezzo/}

\author[S. Mart\'inez]{Sandra Mart\'inez}
\address{S. Mart\'inez \hfill\break\indent IMAS - CONICET and 
Departamento de Matem\'a\-tica, 
\hfill\break\indent
FCEyN - Universidad de Buenos Aires, 
\hfill\break\indent
Ciudad Universitaria, 
\hfill\break\indent
Pabell\'on I  (1428) Buenos Aires, Argentina.}
\email{smartin@dm.uba.ar}

\begin{document}

\begin{abstract}
	The purpose of this work is to study a finite element method for finding solutions to  the eigenvalue problem for the  fractional Laplacian. We prove that the discrete eigenvalue problem converges to the continuous one and we show the order  of such convergence. Finally, we perform some numerical experiments and compare our results with previous work by other  authors.	
\end{abstract}

\maketitle
\section{Introduction and Main Results}

	Anomalous diffusion phenomena are ubiquitous in nature \cite{Klafter,MetzlerKlafter}, and
	the study of nonlocal operators has been an active area of research in different branches of 
	mathematics. Such operators arise in applications as image processing~\cite{BuadesColl,GattoHesthaven,GilboaOsher,YifeiZhang}, finance \cite{CarrHelyette,RamaTankov}, electromagnetic fluids  \cite{McCayNarasimhan}, peridynamics  \cite{Silling}, porous media flow \cite{BensonWheatcraft,CushmanGinn}, among others. 
	
	\medskip
	
	One striking example of a nonlocal operator is the fractional Laplacian $(-\Delta)^s$, defined by
	\[
		(-\Delta)^su(x) \coloneqq  2 C(n,s)
		\int_{\mathbb{R}^n}\dfrac{u(x)-u(y)}{|x-y|^{n+2s}}\,dy,
		\quad x\in\mathbb{R}^n.
	\]
	Here the integral is understood in the principal value sense and the normalization constant $C(n,s)$ is given by
	\[
		C(n,s) \coloneqq \frac{2^{2s-1} s 
		\Gamma(s+\frac{n}{2})}{\pi^{n/2} 
		\Gamma(1-s)} .
	\]
	In the theory of stochastic processes, this operator appears as the infinitesimal generator of a stable L\'evy process, see 
	for instance \cite{Bertoin,Valdinoci}. Moreover, the fractional Laplacian is also one of the simplest examples of a 
	pseudo-differential operator, because its symbol is just $P(\xi)=|\xi|^{2s}.$

	\medskip
	
	An interesting problem concerning the fractional Laplacian is to find its eigenspaces on bounded domains. 
	Namely, to find a positive number $\lambda$ (eigenvalue) and a function $u\not\equiv0$ (eigenfunction) such that
	\begin{equation}\label{eq:int.autovalores}
		\left\lbrace\begin{array}{rl}
			(-\Delta)^s u = \lambda  u &\mbox{ in }\W, \\
				u = 0 &\mbox{ in }\W^c=\mathbb{R}^n\setminus\Omega,
		\end{array} \right.
	\end{equation}
	where $\Omega$ is a bounded domain in $\mathbb{R}^n$ and $s\in(0,1)$.
	Observe that, due to the fact that pointwise values of $(-\Delta)^s u$ depend on the value of $u$ over the 
	whole space, boundary conditions need to be substituted by volume constraints on the complement of $\W$. 
	
{A natural application of the problem we are considering in this paper is given by the fractional Schr\"odinger equation. This equation arises from extending the Feynman path integral approach from Brownian-like quantum mecanical paths --that lead to the classical Schr\"odinger equation-- to L\'evy-like paths \cite{Laskin}. In this regard, eigenfunctions of the fractional Laplacian correspond to the energy states of the system being modeled.
This has motivated researchers to study this problem, both from the physical, mathematical and computational point of view. Among the various references in these subjects, we refer the reader to  \cite{Amore,Antoine,Bao,ChenSong,DuoZhang,Luchko} for further details.
}	
	
	Even if $\Omega$ is an interval, it is very challenging to obtain closed analytical expressions for the eigenvalues and eigenfunctions of the fractional Laplacian. This motivates the utilization of discrete approximations of this problem (see, for example, \cite{Ghelardoni,Kwasnicki,ZoiaRossoKardar}); in this work we consider a finite element method. In first place, we prove that the discrete eigenvalue problem converges to the continuous one. Then, we show the order of convergence for eigenvalues and eigenfunctions, both in the energy and the $L^2$-norm. Orders of convergence are increased by considering suitably graded meshes that stem from a precise characterization of the behavior of eigenfunctions near the boundary of $\W$. Finally, we perform some numerical experiments and compare our results with previous work by other authors. These results are in good agreement with our theory. 
	
The finite element method is flexible enough to deal with non-convex domains, and enables us to provide estimates and sharp upper bounds for eigenvalues even in this context. Moreover, as a consequence of our numerical experiments in the $L$-shaped domain $\W = [-1,1]^2\setminus [0,1]^2$, we conjecture that the first eigenfunction for this domain is as regular as the first one in any smooth domain.

{Due to the nonlocal nature of the problem, a straightforward implementation of the finite element method demands a double loop over the elements to assembly the stiffness matrix. Thus, the complexity of this routine is quadratic with respect to the number of elements. As reported in \cite{ABB}, this step requires about 99\% of the total CPU time. This issue has been tackled in \cite{AG}, where sparse approximations of the stiffness matrix have been proposed and analyzed. Furthermore, the condition number of the stiffness matrix $A$ ($\kappa(A)$) corresponding to the fractional Laplacian of order $s$ using standard piecewise linear finite elements over a mesh with size $h$ scales as $\kappa(A) \simeq h^{-2s}$. Therefore, in the numerical examples performed for this work we have solved the discrete systems by using a direct solver. 
}
	
\subsection*{Main Results} In order to state our results we need to collect some notation and definitions. 
	The natural functional space for the eigenvalue problem \eqref{eq:int.autovalores} is
	\[
		\widetilde{H}^s (\W) \coloneqq \left\{ v \in H^s(\rn) \colon \text{ supp } v \subset \bar{\W} \right\},
	\]
	where $H^s(\R^n)$ is the space of all functions $u\in L^2(\R^n)$ such that
	\[
		|v|_{H^s(\R^n)}^2 \coloneqq \iint_{\R^{2n}} 
				\frac{|v(x)-v(y)|^2}{|x-y|^{n+2s}} \, dx \, dy<\infty. 
	\]
	Moreover, $\left(\V,\|\cdot\|_\V\right)\coloneqq \left(\widetilde{H}^s (\W), \sqrt{C(n,s)} \, |\cdot|_{H^s(\R^n)} \right)$
	is a Hilbert space with the inner product
	\[
			\llangle u, v\rrangle \coloneqq 
			C(n,s) \iint_{\mathbb{R}^{2n}} 
			\frac{(u(x)-u(y))(v(x)-v(y))}{|x-y|^{n+2s}} \, dx \, dy.
	\]
	Obviously, the constant $\sqrt{C(n,s)}$ has no effect on the definition of the space; it is included in order to make the notation simpler in the rest of the paper.
	The fractional space $H^s(\R^n)$ can also be defined for any $s>1.$
	If $s = m + \sigma$, where $m \in\mathbb{N}$ and $\sigma \in (0,1)$, 
	$H^s(\R^n)$ is the space of all functions $v \in H^m(\R^m)$ such that its weak derivatives of order $m$ belong to $H^{\sigma}(\R^n)$. For more details, see Section \ref{sec.preliminares}.
	
	\medskip
	
	In this context, the eigenvalue problem \eqref{eq:int.autovalores} has the following variational formulation:
	find $\lambda\in (0,+\infty)$ and 
	$u\in \V $ such that $u\not\equiv 0$ and  
	\begin{equation} \label{eq:int.autovalores.de.intro}
		\llangle u, v \rrangle = 
		\lambda (u, v)
		\quad \mbox{ for all } v \in \V, 
	\end{equation}
	where  $( \cdot, \cdot )\colon  L^2(\Omega) \times L^2(\Omega) \to \mathbb{R}$ is the bilinear form 
	\[
		( u, v) \coloneqq \int_{\Omega} u(x)v(x) \, dx.  
	\]
		
	It is well-known (see, for example, \cite{MR3002745}) that there is an infinite  sequence of eigenvalues $\{\lambda^{(k)}\}_{k\in\mathbb{N}}$, 
$$
0<\lambda^{(1)}< \lambda^{(2)}\leq \cdots \leq 	\lambda^{(k)}\leq\cdots,\quad  \lambda^{(k)}\to\infty   \mbox{ as }  k\to \infty,
$$
where the same eigenvalue can be repeated several times according to its multiplicity. The corresponding eigenfunctions $u^{(k)}$, normalized by $\|u^{(k)}\|_{L^2(\Omega)}=1$, form a complete orthonormal set in $L^2(\Omega).$ 
See also \cite{MR3089742,ServadeiValdinoci,MR3271254} and the references therein for further details on the fractional eigenvalue problem.
		
	\medskip
		
	Let us now introduce the discrete space. Let  $\T_h$ be a family of triangulations of $\W$ satisfying
		\begin{align}
		& \exists \sigma > 0 \mbox{ s.t. } h_T \leq \sigma \rho_T,  \tag{Regularity} 
			\label{eq:regularity.intro}
	\end{align}
for any element $T \in \mathcal{T}_h$, where $h_T$ is the diameter of $T$ and $\rho_T$ is the radius of the largest ball contained in $T$.
This is the only requirement we need to impose to our family of triangulations.
	
	We consider continuous piecewise linear functions on $\mathcal{T}_h$, namely
	$$
		\V_h \coloneqq \left\{ v \in \V \colon v \big|_T \in \mathcal{P}_1(T) \ \forall T \in \T_h \right\} .
	$$	
Our Galerkin approximation consists in looking for discrete 
	eigenvalues $\lambda_{h}\in \mathbb{R}$ and $u_h\in\V_h$ such that
	$u_h\not\equiv0$ and
	\begin{equation} \label{eq:int.autovalores.di.intro}
		\llangle u_h, v\rrangle=\lambda_h(u_h,v)\quad\forall
		v\in\V_h.
	\end{equation}
	We can order  the discrete eigenvalues of \eqref{eq:int.autovalores.di.intro} as follows
	\[
		0<\lambda_h^{(1)}\leq \lambda_h^{(2)}\leq \cdots \leq 
		\lambda_h^{(k)}\leq \cdots \le\lambda_h^{(\text{dim} \V_h )},
	\]
	where the same eigenvalue is repeated according to its multiplicity. The corresponding eigenfunctions 
	$u^{(k)}_h$  (normalized by $\|u^{(k)}_h\|_{L^2(\Omega)}=1$) form an 
	orthonormal set in $L^2(\Omega).$ 
		
	\medskip
				
	The purpose of this work is to prove the convergence of the 
	discrete problem \eqref{eq:int.autovalores.di.intro}  to the continuous \eqref{eq:int.autovalores.de.intro}.  
	
Our first result concerning the convergence is to determine it in gap distance,
	that is, we prove that the discrete eigenvalue problem  converges to the continuous, 
	following the definition of convergence given in 
	\cite{Kato,Boffia} (see also \cite{Boffi}). More precisely, we define 
	the gap between Hilbert spaces $E, F \subset H$ by
	\[
		\delta(E,F)=\sup_{u\in E,\|u\|_{H}=1} \inf_{v\in F} \|u-v\|_{H},
		\quad 
		\hat{\delta}(E,F)=\max(\delta(E,F),\delta(F,E)).
	\]
	Then, taking $H=\V,$ we say that the discrete eigenvalue problem 
	\eqref{eq:int.autovalores.di.intro} converges to the continuous one 
	\eqref{eq:int.autovalores.de.intro} if, for any $\varepsilon>0$ and $k>0$, there is 
	$h_0>0$ such that 
	\[
		\max_{1\leq i \leq m(k)} |\lambda^{(i)}-\lambda_h^{(i)}|
		\leq \varepsilon,\qquad
		\hat{\delta}\left( \bigoplus_{i=1}^{m(k)}  E^{(i)},
		 \bigoplus_{i=1}^{m(k)}  E_h^{(i)}\right)\leq \varepsilon,
	\]
	for all $h<h_0$. 
	Here, $m(k)$ is the dimension of the space spanned by the first distinct $k$ eigenspaces
	and $E^{(i)}$ and $E^{(i)}_h$ are the eigenspace and the discrete eigenspace associated to $\lambda^{(i)}$
	and $\lambda^{(i)}_h,$ respectively.

\begin{remark} \label{rem:extension}
We remark here that to obtain  convergence in gap distance, we only need to assume that the extension $H^s(\W)\to H^s(\rn)$ is continuous. This is in turn equivalent to the following condition \cite{zhou2015fractional}: there is a constant $C>0$ such that for all $x \in \W$ and all $r \in (0,1]$, 
\begin{equation} \label{eq:extension}
|\W \cap B(x,r)| \ge C r^n \quad \forall x \in \W.
\end{equation} 
\end{remark}		

\begin{theorem}\label{teo:gap} If $\W$ is a fractional extension domain, then 
		the discrete eigenvalue problem 
		\eqref{eq:int.autovalores.di.intro} converges to the continuous one 
		\eqref{eq:int.autovalores.de.intro}.
	\end{theorem}

Having established the convergence of the discrete problem to the continuous,  we next state the order of such a convergence.	
To this end, we need to provide a Sobolev regularity result. This, in turn, requires some additional assumptions on the domain. We prove the following.
	\begin{proposition}\label{prop:regintro} 
		Let $\W \subset \rn$ be a Lipschitz domain satisfying the exterior ball condition and let $u$ be an eigenfunction of $(-\Delta)^s$ in $\W$ with homogeneous Dirichlet boundary conditions. Then, 
		$u \in \widetilde{H}^{s+\nicefrac12-\eps}(\W)$ for any $\eps > 0.$
		
		Moreover, considering the weighted Sobolev scale (cf. \eqref{eq:weighted_sobolev} below) it also holds that $u \in H^{1+s-\eps}_{\nicefrac12 -2\eps}(\rn)$ for any $\eps > 0.$
	\end{proposition}
	
The regularity in standard spaces in the previous proposition is utilized to prove an a priori error bound for the finite element approximations with meshes satisfying \eqref{eq:regularity.intro}. Moreover, the weighted regularity above enables to consider suitably graded meshes (see the definition \eqref{eq:H} in Subsection \ref{ss:fe}), and these deliver an enhanced order of convergence. 
Upon proving approximation properties of the discrete spaces considered, an application of Babu\v ska-Osborn theory \cite{BO91} allows to deduce the rate of convergence for the eigenvalues and for the eigenfunctions in the energy norm, both for uniform and graded meshes.	
	\begin{theorem}\label{teo:ordenenergia}
	Let $\W \subset \rn$ be a Lipschitz domain satisfying the exterior ball condition and let $\lambda^{(k)}$ be an eigenvalue of multiplicity $m$ (that is, $\lambda^{(k)}=\lambda^{(k+1)}=\cdots=\lambda^{(k+m-1)}$ and $\lambda^{(i)} \neq \lambda ^{(k)}$ for $i\neq k,\dots,k+m-1 $). Consider the Galerkin approximations given by \eqref{eq:int.autovalores.di.intro} on a shape-regular familiy of meshes.
\begin{enumerate} 
	\item For any $\varepsilon>0,$ there exists a positive constant $C$ independent of $h$ 
		such that
		\[
			0\leq \lambda_h^{(j)}-\lambda^{(k)}\leq C h^{1-\varepsilon} \quad \forall k\leq j\leq k+m-1.
		\]
		
Moreover, if $u^{(k)}$ is an eigenfunction associated to $\lambda^{(k)}$, there is
		\[
		\{ w_h^{(k)} \} \subset E_h^{(k)} \oplus \ldots \oplus E_h^{(k+m-1)}
		\]
		such that
		\[
			\|u^{(k)}-w^{(k)}_h\|_{\V}\leq C h^{\nicefrac12-\varepsilon}.
		\]
		
	\item On the other hand, if $s>\nicefrac12$ and the meshes are graded according to \eqref{eq:H} with $\mu = 2$, then the estimates above can be refined to be
	\[
			0\leq \lambda_h^{(j)}-\lambda^{(k)}\leq C h^{2-\varepsilon} \quad \forall k\leq j\leq k+m-1.
		\]
and
		\[
			\|u^{(k)}-w^{(k)}_h\|_{\V}\leq C h^{1-\varepsilon}.
		\]		
		\end{enumerate}
	\end{theorem}

Finally, to prove convergence orders in the $L^2$ norm, we require smoothness on the domain. This regularity assumption is required in order to apply an Aubin--Nitsche duality argument. We obtain the following.

\begin{theorem}\label{teo:ordenL2}
Assume $\W$ is a smooth domain and let  $\alpha = \min\{ s, 1/2 - \eps \}$ for any $\eps >0$. Then,
 if $\lambda^{(k)}$ is an eigenvalue of multiplicity $m$ and if $u^{(k)}$ is an eigenfunction associated to $\lambda^{(k)}$, there is
		\[
		\{ w_h^{(k)} \} \subset E_h^{(k)} \oplus \ldots \oplus E_h^{(k+m-1)}
		\]
		such that
		\begin{equation}\label{eq:convL2}
					\|u^{(k)}-w^{(k)}_h\|_{L^2(\Omega)}\leq C h^{\alpha+\nicefrac12-\eps}.	
		\end{equation}
\end{theorem}

In order to illustrate the convergence estimates obtained in Theorems \ref{teo:ordenenergia} and \ref{teo:ordenL2}, we present the results of numerical tests for finite element discretizations of one and two-dimensional eigenvalue problems. Moreover, in the latter case, some examples in domains that do not satisfy the hypotheses of Proposition 
\ref{prop:regintro} are displayed. These examples provide numerical evidence that the assertion of this proposition still holds true under weaker assumptions about the domain.

\subsection*{The paper is organized as follows} Section \ref{sec.preliminares} collects the notation we employ, and reviews some previous works on the problem under consideration. In particular, regularity of eigenfunctions is proved. The section concludes with a discussion of certain aspects of finite element approximations of the fractional Laplacian. Afterwards, Section \ref{sec.conv} deals with the convergence of the discrete eigenvalue problem to the continuous one in gap distance. 
In Section \ref{sec:orders}, proof of the orders of convergence for eigenvalues and eigenfunctions are given, including estimates for graded meshes. Finally, numerical experiments are discussed in Section \ref{sec:numerico}.

\section{Preliminaries and Definitions}
\label{sec.preliminares}
	In this section we review the basic aspects of the problem under consideration. In first place, we set notation regarding 
	Sobolev spaces. Afterwards, we analyze theoretical properties of the eigenvalue problem \eqref{eq:int.autovalores}. Regularity 
	results for weak solutions of fractional Laplace equations are recalled next. The section concludes with the introduction of 
	the finite element spaces we work with.

\subsection{Sobolev spaces}	
	Given an open set $\W \subset \rn$ and $s \in(0,1)$, define the fractional Sobolev space $H^s(\W)$ as
	\[
		H^s(\W) \coloneqq \left\{ v \in L^2(\W) \colon |v|_{H^s(\W)} < \infty \right\},
	\]
	where $|\cdot|_{H^s(\W)}$ is the seminorm
	\[
		|v|_{H^s(\W)}^2 \coloneqq \iint_{\W^2} 
		\frac{|v(x)-v(y)|^2}{|x-y|^{n+2s}} \, dx \, dy. 
	\]
	Obviously, $H^s(\W)$ is a Hilbert space endowed with the norm $\|\cdot\|_{H^s(\W)} = \|\cdot\|_{L^2(\W)} + |\cdot|_{H^s(\W)} .$
	If $s>1$ and it is not an integer, the decomposition $s = m + \sigma$, where $m \in \mathbb{N}$ and $\sigma \in (0,1)$, 
	allows to define $H^s(\W)$ by setting
	\[
		H^s(\W) \coloneqq \left\{ v \in H^m(\W) \colon |D^\alpha v|_{H^{\sigma}(\W)} < \infty \text{ for all } \alpha \text{ s.t } 
		|\alpha| = m \right\}.
	\]
 
	Let us also define the space of functions supported in $\W$, 
	\[
		\widetilde{H}^s (\W) \coloneqq \left\{ v \in H^s(\rn) \colon \text{ supp } v \subset \bar{\W} \right\}.
	\]	
	For $0\le s \le 1$ and if $\W$ is a Lipschitz domain, this space may be defined through interpolation,
	\[ 
		\widetilde{H}^s (\W) = \left[L^2(\W), H^1_0(\W) \right]_{s}.
	\]
	Moreover, depending on the value of $s$, different characterizations of this space are available (see, for example 
	\cite[Chapter 11]{LionsMagenes}). If $s<\nicefrac12$ the space $\widetilde{H}^s (\W)$ coincides with $H^s (\W)$, and if $s>\nicefrac12$ it may 
	be characterized as the closure of $C^\infty_0(\W)$ with respect to the $|\cdot|_{H^s (\W)}$ norm. In the latter case, it is also customary to denote it by $H^s_0(\W)$. The particular case of $s=\nicefrac12$ gives raise to the Lions-Magenes space $H^{\nicefrac12}_{00}(\W)$, 
	which can be characterized by
	\[
		H^{\nicefrac12}_{00} (\W) \coloneqq \left\{ v \in H^{\nicefrac12}(\W) \colon \int_\W \frac{v(x)^2}{\text{dist}(x,\partial \W)} \, dx < 
		\infty \right\}.
	\]
	Note that the inclusion  $H^{\nicefrac12}_{00}(\W) \subset H^{\nicefrac12}_{0}(\W) = H^{\nicefrac12}(\W)$ is strict.

	It is apparent that $\llangle \cdot, \cdot\rrangle$
	defines an inner product on $\widetilde{H}^s(\W).$ In addition,
	the norm induced by it, which is just 
	the $H^s(\rn)$ seminorm, is equivalent to the 
	full $H^s(\rn)$ norm on this space, because of the following well known result.
	See for instance \cite[Lemma 2.5]{MR3556755}.

	\begin{proposition}[Poincar\'e inequality] 
		Let $\W$ be a bounded domain, then there is a constant $c=c(\W,n,s)$ such that 
		\begin{equation} \label{eq:poincare}
			\| v \|_{L^2(\W)} \leq c |v|_{H^s(\rn)}  \quad \forall v \in 
			\widetilde{H}^s(\W) .
		\end{equation}
	\end{proposition}

     For the proof of the following  result, see e.g. \cite{MR2895178,Hitchhikers}.

\begin{proposition} \label{prop:compact}
    Let $\W$ be an extension domain (cf. Remark \ref{rem:extension}). Then, the inclusion 
    $\widetilde{H}^s(\W) \hookrightarrow L^2(\R^n)$ is compact.
\end{proposition}

\medskip 

Weighted spaces are a customary tool when dealing with singular solutions. 
As in \cite{AcostaBorthagaray}, we define the following weighted fractional Sobolev spaces. The weights we consider are powers of the distance to the boundary of $\Omega$. We introduce the notation  
\begin{equation} \label{eq:def_delta2}
\delta(x,y) = \min \{ \text{dist}(x,\partial \W), \text{dist}(y,\partial \W) \}.
\end{equation} 
Let $s = m + \sigma$, with $m \in \N$ and $\sigma \in (0,1)$, then 
\begin{equation} \label{eq:weighted_sobolev}
H^{s}_\alpha (\W) = \left\{ v \in H^m (\W) \colon | D^\beta v |_{H^{\sigma}_\alpha (\W)} < \infty \ \forall \beta \in \N^n \mbox{ s.t. } |\beta| = m \right\} ,
\end{equation}
where 
$$
| w |_{H^{\sigma}_\alpha (\W)} = \iint_{\W\times\W} \frac{|w(x)-w(y)|^2}{|x-y|^{n+2\sigma}} \, \delta(x,y)^{2\alpha} dx \, dy.
$$

We equip this space with the norm    
$$
\| v \|_{H^{s}_\alpha (\W)}^2 = \| v \|_{H^m (\W)}^2 + \sum_{|\beta| = m } 
| D^\beta v |_{H^{\sigma}_\alpha (\W)} . $$

We also need to define spaces over $\rn$. The global weighted Sobolev space $H^s_{\alpha, \W}(\rn)$ is
\begin{equation*} 
H^{s}_{\alpha, \W}(\rn) = \left\{ v \in H^m (\rn) \colon | D^\beta v |_{H^{\sigma}_{\alpha,\W} (\rn)} < \infty \ \forall \beta \in \N^n \mbox{ s.t. } |\beta| = m \right\} ,
\end{equation*}
where 
$$
| w |_{H^{\sigma}_{\alpha,\W} (\rn)} = \iint_{\rn\times\rn} \frac{|w(x)-w(y)|^2}{|x-y|^{n+2\sigma}} \, \delta(x,y)^{2\alpha} dx \, dy.
$$
The norm on this space is  
$$
\| v \|_{H^{s}_{\alpha,\W} (\rn)}^2 = \| v \|_{H^m (\rn)}^2 + \sum_{|\beta| = m } 
| D^\beta v |_{H^{\sigma}_{\alpha,\W} (\rn)} . $$
Whenever the set $\W$ is clear from the context,  we drop the reference to it in the global case and simply write $H^s_{\alpha} (\rn)$.

\begin{remark}
 \label{rem:pesosA2}
Although  we are interested in the case $ \alpha\ge 0$, we recall that in the definition of weighted Sobolev spaces $H^m_\alpha(\W)$, with $m$ being a nonnegative integer, arbitrary powers of $\delta(x)$ can be considered \cite[Theorem 3.6]{Kufner}. On the other hand, for general  weights some restrictions must be taken into account in order to get an adequate definition of the spaces, namely, to ensure their completeness. A classical family of weights is that of the Muckenhoupt $A_2$ class. In the global version $H^s_{\alpha} (\rn)$ we need to restrict the range of $\alpha$ to $|\alpha| < 1/2$ in order to have $\delta^{2\alpha}\in A_2$. 
\end{remark}


\subsection{Eigenvalue problem}
	In the sequel, we work within the Hilbert space
	$$
		(\V, \| \cdot \|_\V) \coloneqq (\widetilde{H}^s(\W), \sqrt{C(n,s)} \, |\cdot|_{H^s(\rn)}).
	$$ 
	
	In \cite{MR3002745}, the authors prove that for any $k\in\mathbb{N}$
	the eigenvalues of \eqref{eq:int.autovalores.de.intro} can be characterized as follows:
	\[
		\lambda^{(k)}=
		\min\left\{\dfrac{\|u\|_{\V}^2}{\|u\|_{L^2(\Omega)}^2}
		\colon u\in \V^{(k)}\setminus
		\{0\}
		\right\},
	\]
	where $\V^{(1)}=\V$ and
	\[
		\V^{(k)}\coloneqq\left\{u\in\V\colon
		\llangle u,u^{(j)}\rrangle=0\quad\forall j=1,\dots,k-1\right\}
	\]
	for all $k\ge2.$ Therefore, by the min-max theorem,
	\[
		\lambda^{(k)} = \min_{E \in S^{(k)}} \max_{u \in E} 
		\frac{\| u \|^2_\V}{\| u \|^2_{L^2(\W)}}
	\]
	where $S^{(k)}$ denotes the set of all $k-$dimensional subspaces of 
	$\V.$

The first eigenvalue $\lambda^{(1)}$ is simple (see, for example, \cite{MR3002745}). We now state some regularity properties of the eigenfunctions.

\subsection{Regularity results} \label{ss:regularity}
	Given a function $f \in H^r(\W)$ ($r \ge -s$), let us consider the homogeneous Dirichlet problem for the fractional
	Laplacian,		
		\begin{equation}\label{eq:fuente}
			\left\lbrace \begin{array}{rl}
				(-\Delta)^s u = f  &\mbox{ in }\W, \\
				u = 0 &\mbox{ in }\W^c .
			\end{array} \right.	
		\end{equation}
	Existence and uniqueness of a weak solution $u \in \widetilde{H}^s(\W)$ of the above equation is an immediate consequence of the Lax-Milgram lemma. 
	We are interested in regularity estimates for such solution in standard and graded Sobolev spaces. In \cite{AcostaBorthagaray}, these are obtained in terms of the H\"older regularity of the data.

\begin{proposition}[See \cite{AcostaBorthagaray}] \label{prop:ab}
Let $\W$ be a Lipschitz domain satisfying the exterior ball condition and consider $\beta = \nicefrac12 - s$ if $s<\nicefrac12$ or $\beta>0$ if $s\ge\nicefrac12$. Then, if $f\in C^{\beta}(\W)$ for every $\eps >0$, the solution $u$ of \eqref{eq:fuente} belongs to $\widetilde H^{s+\frac{1}{2}-\eps}(\W)$, with
$$\|u\|_{\widetilde H^{s+\frac{1}{2}-\eps}(\W)} \leq \frac{C(\W,s,n)}{\eps} \, \|f\|_{C^{\beta}(\W)}.$$

Moreover, if $s > \nicefrac12$ and $f \in C^{1-s}(\W)$, then for every $\eps >0$, it holds that $u \in H^{s+1-2\eps}_{\frac12 - \eps}(\W)$, with
$$\|u\|_{H^{s+1-2\eps}_{1/2 - \eps}(\W)} \leq \frac{C(\W,s,n)}{\eps} \, \|f\|_{C^{1-s}(\W)}.$$
	\end{proposition}

On the other hand, smoothness of eigenfunctions is deduced from the regularity theory for the fractional Laplacian.
See \cite{MR2494809, RosOtonSerra, Servadei20132445, MR3161511, MR2270163}.

\begin{proposition}\label{prop:cinfauto}
If $\W$ is a Lipschitz domain satisfying 
	the exterior ball condition then any solution of \eqref{eq:int.autovalores} is in $C^{\infty}(\Omega)\cap L^\infty(\Omega)$.
\end{proposition}

Sobolev regularity of eigenfunctions is a consequence of the two previous propositions.
\begin{proof}[Proof of Proposition \ref{prop:regintro}]
Since $u \in C^\infty(\W)$ (cf. Proposition \ref{prop:cinfauto}), the claim follows easily applying Proposition \ref{prop:ab}. 
	\end{proof}
	
	Following Grubb \cite{Grubb}, it is also possible to obtain Sobolev regularity results for the solution to \eqref{eq:fuente} in terms of Sobolev regularity of the right hand side. In that paper, the author deals with	H\"ormander $\mu-$spaces $H^{\mu(s)}_p$; see that work for a definition and
	further details. The following result is a particular case of Theorem 7.1 therein:
	\begin{theorem}\label{teo:Grubb}
		Let $\W$ be a smooth domain, $\ell>s-\nicefrac12$ and assume 
		$u \in \widetilde{H}^{\sigma}(\W)$ for some $\sigma>s-\nicefrac12$ and consider a right hand side function
		$f\in H^{\ell-2s} (\W)$. Then, it holds that $u \in H^{s(\ell)}(\W)$.
	\end{theorem}
 
	In particular, considering $\ell = r+2s$ in the previous theorem and taking into account that 
	\[ 
		H^{s(r+2s)}(\overline \W)
		\begin{cases}
				= \widetilde{H}^{2s+r}(\W)  &\mbox{ if } 0 < s+r < \nicefrac12, \\
				\subset  \widetilde{H}^{s+\nicefrac12-\eps}(\W) \ \forall \eps > 0, &\mbox{ if } 
				\nicefrac12 \le s+r < 1,
			\end{cases}	
	\]
	(see \cite[Theorem 5.4]{Grubb}), we obtain:

	\begin{proposition}\label{prop:regHr}
		Let $\W$ be a smooth domain, $f\in H^r(\Omega)$ for $r\geq -s$ and $u\in \widetilde{H}^s(\W)$ be the solution of the Dirichlet problem
		\eqref{eq:fuente}. Then, the following regularity estimate holds 
		$$
			|u|_{H^{s+\alpha}(\rn)} \leq C(n, \alpha) \|f\|_{H^r(\W)}.
		$$
		Here, $\alpha = s+r$ if $s+r < \nicefrac12$ and $\alpha = 
		\nicefrac12 - \eps$ if $s+r \ge \nicefrac12$, with $\eps > 0$ arbitrarily small.
	\end{proposition}

	\begin{remark}\label{remark:bdry}
		Assuming further Sobolev regularity in the right hand side function does not imply that the solution will 
		be any smoother than what is given by the previous proposition. Indeed, if $f \in H^r(\W)$, then 
		Theorem \ref{teo:Grubb} gives $u \in H^{s(r+2s)}(\W)$, which can not be embedded 
		in any space sharper than $H^{s+\nicefrac12-\eps}(\W)$ if $r+s\ge \nicefrac12$.

		Moreover, other regularity estimates for eigenfunctions of the fractional Laplacian in smooth domains are derived in 
		\cite{Grubb_autovalores,MR3294242}. 
		These estimates are formulated in terms of H\"older norms. Letting $d$ be a smooth function that behaves like
		$\text{dist}(x,\partial \W)$ near the boundary of $\W$, it is shown that any eigenfunction $u$ of 
		\eqref{eq:int.autovalores} lies in the space $d^sC^{2s(-\eps)}(\overline \W)$, where the $\eps$ is active only if 
		$s=\nicefrac12$ and that $\nicefrac{u}{d^s}$ does not vanish near $\partial\W$. 
		This shows that no further regularity than $H^{s+\nicefrac12-\eps}(\W)$ should be expected for eigenfunctions. 
	\end{remark}

\subsection{Finite element approximations} \label{ss:fe}

	Let $\T_h$ be a family of shape-regular triangulations  of $\W$ (see introduction).
	Observe that for all $h>0$ the discrete space $\V_h$ is a subset of the continuous space.

	We have the analogue  min-max characterization for eigenvalues of the discrete problem, 
	\[
		\lambda_h^{(k)} = \min_{E \in S_h^{(k)}} \max_{u \in E} 
		\frac{\| u \|^2_\V}{\| u \|^2_{L^2(\W)}},
	\]
	where $S_h^{(k)}$ denotes the set of all
	$k$ dimensional subspaces of $\V_h.$

	\medskip

	\begin{remark}\label{re:compara}
		It follows from $\V_h \subset \V$ that
		\[
			\lambda^{(k)}\leq \lambda ^{(k)}_h .
		\] 
	\end{remark}

The second part of Proposition \ref{prop:regintro} is exploited by using finite element approximations on adequately graded meshes. This idea is standard, in problems with corner singularities or to cope with boundary layers arising in convection-dominated problems.
The following construction of graded meshes is based on \cite[Section 8.4]{Grisvard}. We assume that in addition to being shape-regular, our sequence of meshes enjoys satisfies the following graded hypotheses. First, we pick an arbitrary mesh size parameter $h>0$ and define, for $\varepsilon$ small enough, a number $1\le \mu$. 
Then,  we assume that for any $T\in \mathcal{T}_h$,
\begin{equation} \label{eq:H} \tag{H}
 \begin{array}{ll}
\mbox{if } T\cap \partial \W\neq \emptyset, &
 \mbox{then } h_T\le C(\sigma) h^{\mu}; \\
\mbox{otherwise, } & h_T\le C(\sigma) h \,\mbox{dist}(T,\partial \Omega)^{(\mu-1)/\mu}. 
 \end{array}
 \end{equation} 
Constructing graded meshes as above, finite element approximations to solutions of \eqref{eq:fuente} were proved to deliver an enhanced order of convergence (see \cite{AcostaBorthagaray}).

\medskip

Lastly, we want to mention that in the following sections we will consider  a quasi-interpolation operator
	$I_h\colon H^l(\W) \to \V_h $ satisfying the following estimate: there exists $C>0$ such that 
	for any $w\in H^{l}(\W)$,
	\begin{equation}
		\label{estiminterpol}
		\|w - I_h w\|_{\V} \le C h^{l-s}  |w|_{H^{l}(\W)} .
	\end{equation}
	Quasi-interpolation operators were introduced in \cite{clement} (see also \cite{ScottZhang}), and an estimate like 
	\eqref{estiminterpol} is derived, for example, in \cite{AcostaBorthagaray}.

\section{Convergence of eigenvalues and eigenfunctions in gap distance}
\label{sec.conv}

	In this section we prove Theorem \ref{teo:gap}, stating that the discrete eigenvalue problem 
	\eqref{eq:int.autovalores.di.intro}  converges to the continuous
	\eqref{eq:int.autovalores.de.intro} in the gap distance (recall the definition of convergence
	given in the introduction). We only assume that the domain $\W$ satisfies \eqref{eq:extension}, so that 
	the embedding $\V \hookrightarrow L^2(\rn)$ is compact (cf. Proposition \ref{prop:compact}).

	Let us start by defining the solution operators of the continuous 
	and discrete problems, $T:L^2(\Omega)\to \V$ and 
	$T_h:L^2(\Omega)\to \V_h$. Given  $f\in L^2(\Omega)$, 
	we define $Tf\in \V$ as the  unique solution of 
	\begin{equation}\label{ecuT} 
		\llangle Tf,v \rrangle =( f, v) \quad \forall v\in \V,
	\end{equation}
 	and    $T_h f\in \V_h$ as the  unique solution of 
 	\[
 		\llangle T_h f,v_h \rrangle =(f,v_h) \quad \forall v_h\in \V_h. 
	\]
	Observe that if $(u,\lambda)$ is an eigenpair, then $T(\lambda u) = u$ and $T_h (\lambda u) = \Pi_h u$. 
	
	To prove Theorem \ref{teo:gap}, by \cite[Proposition 7.4 and Remark 7.5]{Boffi}, we only need to show that
	the operators $T$ and $T_h$ are compact and
	\[
		\| T- T_h\|_{\mathcal{L}(L^2(\Omega),\V)} \to 0 \mbox{ as } h \to 0.
	\]
	
	\medskip 
	
	\begin{lemma}\label{compacto}
		The operators $T$ and $T_h$ are compact.
	\end{lemma}
	\begin{proof}
		Let $\{f_k\}_{k\in\mathbb{N}}$ 
		be a bounded sequence in $L^2(\Omega).$  
		Then, there exists a subsequence of  $\{f_k\}_{k\in\mathbb{N}}$ 
		(still denoted by $\{f_k\}$) and $f\in L^2(\Omega)$ such that  
		$f_k \rightharpoonup f$ weakly in $L^2(\Omega).$ 
		Taking $v=T f_k$ in \eqref{ecuT},  we get
		\[ 
			\|T f_k\|^2_{\V}=(f_k, T f_k)\leq C \|T f_k\|_{L^2(\Omega)}
			\quad\forall k\in\mathbb{N}.
		\]
		Therefore, by Poincar\'e inequality \eqref{eq:poincare},  we have that 
		$\{T f_k\}_{k\in\mathbb{N}}$ is bounded in $\V$. 
		Thus, there exists a subsequence of  $\{f_k\}_{k\in\mathbb{N}}$
		(still denoted by $\{f_k\}_{k\in\mathbb{N}}$) and $u\in\V$
		such that  $T f_k \rightharpoonup u$ weakly in $\V.$ Hence
		\[
			\llangle u,v \rrangle =\lim_{k\to\infty} \llangle Tf_k,v \rrangle =\lim_{k\to\infty}( f_k, v)
			=( f, v)\quad\forall v\in\V,
		\]
		that is, $u=Tf.$
		
		On the other hand, since the inclusion $\V\hookrightarrow L^2(\R^n)$ is compact,
		passing if necessary to a subsequence we may assume
		\begin{align*}
			T f_k \rightharpoonup Tf &\mbox{ weakly in } \V,\\
			T f_k \rightarrow Tf & \mbox{ strongly in } L^2(\R^n).
		\end{align*}
		Then,
		\[ 
			\|T f_k\|_{\V}^2=(f_k, T f_k) \to (f, T f)=\|T f\|^2_{\V}
		\]
		as $n\to\infty.$ Since the space $\V$ is uniformly convex, it follows that
		\[
			T f_k \rightarrow Tf \mbox{ strongly in } \V.
		\]
		
		For the operators $T_h$ the result follows since
	 they are  finite-rank operators. 
	\end{proof}
 
	 \begin{lemma}\label{convop} 
		The following norm convergence holds true:
		\[
			\| T- T_h\|_{\mathcal{L}(L^2(\Omega),\V)} \to 0 \mbox{ as } h \to 0.
		\]
 	\end{lemma}

	\begin{proof}
 		For each $h$, take $f_h\in L^2(\Omega) $ 
		such that $\|f_h\|_{L^2(\Omega)}=1$ and
		\[
			\sup_{\|f\|_{L^2(\Omega)}=1} \| T f - T_h f\|_{\V}
			= \| T f_h - T_h f_h\|_{\V}.
		\]
		
		Then, to prove the result, it is enough to show that for any sequence $h_k\to0$
		there is a subsequence $\{h_{k_j}\}_{j\in\N}$ such that 
		\[
			\| T f_{h_{k_j}} - T_{h_{k_j}} f_{h_{k_j}}\|_{\V}\to 0 \quad\text{ as } j\to\infty.
		\]
		
		Let $\{h_k\}_{k\in\N}$ be a sequence sucht that $h_{k}\to0.$
 		It follows from $\|f_{h_k}\|_{L^2(\Omega)}=1$ for all $k\in\N$ that  
 		there exist a subsequence $\{f_{h_{k_j}}\}_{j\in\mathbb{N}}$ 
		of $\{f_{h_k}\}_{k\in\N}$ and $f\in L^2(\Omega)$ such that  
		$ f_ {h_{k_j}}  \rightharpoonup  f$ weakly in $L^2(\Omega)$. 
		Proceeding as in the proof of Lemma \ref{compacto} 
		(and passing if necessary to a subsequence), we may assume
 		\begin{align*}
			T_{h_{k_j}} f_{h_{k_j}} \rightharpoonup v &\mbox{ weakly in } \V, \\
			T_{h_{k_j}} f_{h_{k_j}} \rightarrow v &\mbox{ strongly in } L^2(\R^n).
		\end{align*}
		
		On the other hand, it follows from \eqref{estiminterpol} that 
		\begin{align*}
			I_h \varphi  \rightarrow \varphi &\mbox{ strongly in } \V, \\
			I_h \varphi  \rightarrow \varphi &\mbox{ strongly in }  L^2(\R^n),
		\end{align*}
		for any $\varphi\in C^{\infty}_0(\Omega).$ Therefore,  
		\[
			\llangle v,\varphi \rrangle =\lim_{j\to \infty} \llangle T_{h_{k_j}} f_{h_{k_j}}, 
				I_{h_{k_j}} \varphi \rrangle =\lim_{j \to \infty }
				( f_{h_{k_j}} ,I_{h_{k_j}} \varphi)
				=( f , \varphi)\quad\forall\varphi\in C^{\infty}_0(\Omega),
		\] 
		which means that $v=T f$. 
		Then,
		\[
			\| T f_{h_{k_j}} - T_{h_{k_j}} f_{h_{k_j}} \|^2_{\V}=( f_{h_{k_j}}, T f_{h_{k_j}} - T_{h_{k_j}} f_{h_{k_j}}) \to 0.
		\]
	 \end{proof}
	Now we conclude the convergence of the discrete eigenvalue problem to the continuous in the gap distance.

	\begin{proof}[Proof of Theorem \ref{teo:gap}]
	      The proof follows by  Lemmas \ref{compacto} and \ref{convop} 
	      and using Proposition 7.4 and Remark 7.5 of \cite{Boffi}.
	\end{proof}

\section{Order of convergence} \label{sec:orders}
	Assuming certain regularity on the domain $\W$, we are able to deduce orders 
	of convergence of the finite element approximations.
	This is attained as an application of the Babu\v ska-Osborn theory \cite{BO91};  
	an important tool in this regards is given by considering approximation properties of 
	$\Pi_h \colon \V \to \V_h$, the projection with respect to 
	the $\| \cdot \|_{\V}$ norm. To be specific, given $u \in \V$,  this is the only function in $\V_h$ such that the Galerkin orthogonality
	\[
			\llangle u - \Pi_h u , \, v_h \rrangle = 0 \quad\forall v_h \in \V_h
	\]
	holds, or equivalently, 
	\begin{equation}\label{eq:cea}
		\| u - \Pi_h u \|_{\V} = \inf_{v_h \in \V_h} \|u - v_h \|_{\V}.
	\end{equation}
	Observe that if $u$ is the solution of \eqref{eq:fuente}, 
	then $\Pi_h u$ corresponds to the solution of the corresponding
	discrete problem on $\V_h$.

	\begin{proposition}\label{estimacionlinf0}  
		Let $\W$ be a Lipschitz domain satisfying the exterior ball condition and
		$u$ be an eigenfunction of \eqref{eq:int.autovalores.de.intro}. Then, 
		for any $\varepsilon>0$ there exists a positive constant $C$ independent of $h$ 
		such that
		\begin{equation}
			\label{estimacionlinf1}
				\|u-\Pi_h u\|_{\V} \le C h^{\nicefrac12-\eps}.
		\end{equation}
Also, if $s > \nicefrac12$, constructing meshes according to grading hypothesis \eqref{eq:H} (setting the parameter $\mu$ equal to $2$), it holds that
		\begin{equation}
		\| u - \Pi_h u \|_\V  \leq  C h \sqrt{|\ln h|}. 
		\label{eq:conv_graduadas}
		\end{equation}
	\end{proposition}	
 	\begin{proof}
	
Upon considering estimate \eqref{eq:cea} and Proposition \ref{prop:regintro}, the proof follows as in \cite[Theorem 4.7]{AcostaBorthagaray}.
	\end{proof}

\begin{remark}
In view of the previous proposition, and applying the abstract theory from \cite{BO91} Theorem \ref{teo:ordenenergia} follows.
\end{remark} 
	
In the remainder of this section, we study convergence of discrete eigenfunctions in the $L^2$ norm. 
We first prove the $L^2$ convergence of the energy projection 
over the discrete space. Notice that smoothness of the domain is required
in order to apply Proposition \ref{prop:regHr}.

\begin{proposition}\label{propl2} Let $\W$ be a smooth domain and $u$
 be an eigenfunction of \eqref{eq:int.autovalores.de.intro}. Then, for any $\eps > 0$ there
is a positive constant $C$ independent of h such that
\begin{equation}\label{eq:aproxL2}
			\|u-\Pi_h u\|_{L^2(\W)} \le	C h^{\nicefrac12+\alpha-\eps}.
		\end{equation}
		Here, $\alpha = s$ if $s < \nicefrac12$ and $\alpha = \nicefrac12 - \eps$ if $s\ge 
		\nicefrac12.$
\end{proposition}
\begin{proof}
	We apply an Aubin--Nitsche duality argument. 	
		Let $w\in\V$ be the weak solution of the boundary value problem
			\[
			     \left\lbrace \begin{array}{cl}
				    (-\Delta)^sw=u-\Pi_h u &\mbox{ in }\Omega,\\
				    w=0&\mbox{ in }\Omega^c.\\
			      \end{array} \right.
			\]
		Then, resorting to Galerkin orthogonality again we obtain
			\[  		 
			      \|u-\Pi_h u\|_{L^2(\W)}^2 = 
			      \llangle  w, u-\Pi_h u \rrangle 
			      \leq \|w - I_h w\|_{\V} \|u-\Pi_h u\|_{\V},
			\]
		where $I_h w \in \V_h$ is the interpolator of
			$w.$
			 
		Taking into account the regularity given by Proposition \ref{prop:regHr} with $r=0$, interpolation estimate \eqref{estiminterpol} gives
			\[
			      \|w - I_h w\|_{\V} \le C h^{\alpha}  |w|_{H^{s+\alpha}(\W)} 
			      \le C  h^{\alpha} \|u-\Pi_h u\|_{L^2(\W)}.
			\]
		Finally, using the error estimate \eqref{estimacionlinf1} we obtain 
		\begin{equation}\label{estimacionlinf} 		 
			 \|u-\Pi_h u\|_{L^2(\W)}^2 \le C h^{\nicefrac12+\alpha-\eps}|u|_{H^{s+\nicefrac12-\varepsilon}(\W)}\|u-\Pi_h u\|_{L^2(\W)},
			 \end{equation}
		and then estimate \eqref{eq:aproxL2} follows.
\end{proof}

The proof of Theorem \ref{teo:ordenL2} follows as in \cite[Lemma 6.4-3]{RT83}, using Proposition \eqref{propl2}. We include a proof here for completeness. 

\begin{proof}[Proof of Theorem \ref{teo:ordenL2}]
		In first place, we assume that $m=1$ since the case $m>1$ is similar (see \cite{Boffi,RT83}). We  define
		$$
			\w_h^{(k)} \coloneqq \left( \Pi_h u^{(k)}, u_h^{(k)} \right) u_h^{(k)} ,
		$$
		and the quantity
		$$
			\rho_h^{(k)} \coloneqq \max_{i \neq k} \frac{\lambda^{(k)}}{|\lambda^{(k)} - \lambda_h^{(i)}|}.
		$$

		Then
		\begin{equation}
		    \begin{split}
			\|u^{(k)} - u_h^{(k)} \|_{L^2(\W)} 
			&\leq \|u^{(k)} - \Pi_h u^{(k)} \|_{L^2(\W)} \\
			&\quad
			+ \| \Pi_h u^{(k)} - \w_h^{(k)} \|_{L^2(\W)} + 
			\| \w_h^{(k)} - u_h^{(k)}\|_{L^2(\W)} .
			    \end{split}
			\label{eq:triangular}
		\end{equation}
		We are going to estimate the terms in the right hand side separately. 
		
		Given $\varepsilon>0,$ it follows from our regularity estimate \eqref{eq:aproxL2} that
		there exists $C>0$ independent of $h$ such that 
		\begin{equation}\label{eq:des1}
				\| u^{(k)} - \Pi_h u^{(k)} \|_{L^2(\W)} \leq C h^{\alpha+\nicefrac12-\eps} .
		\end{equation}
		Moreover, since 
		$$
			\left( \Pi_h u^{(k)}, u_h^{(i)} \right) = \frac{1}{\lambda_h^{(i)}} \llangle \Pi_h u^{(k)}, u_h^{(i)} \rrangle = 
			\frac{1}{\lambda_h^{(i)}} \llangle u^{(k)}, u_h^{(i)} \rrangle = \frac{\lambda^{(k)}}{\lambda_h^{(i)}} \left( u^{(k)}, 
			u_h^{(i)} \right), 
		$$
		we have
		\[
			\left| \left( \Pi_h u^{(k)}, u_h^{(i)} \right) \right| \leq \rho_h^{(k)} \left| \left( u^{(k)} - \Pi_h u^{(k)}, 
			u_h^{(i)} \right) \right|.
		\]

		So,  
		\begin{equation}\label{eq:des2}
			\begin{split}
				\| \Pi_h u^{(k)} - \w_h^{(k)} \|_{L^2(\W)}^2 
				& = \sum_{i\neq k} 
				\left( \Pi_h u^{(k)}, u_h^{(i)} \right)^2\\ 
				&\leq \left[\rho_h^{(k)}\right]^{2} 
				\sum_{i\neq k} 
				\left( u^{(k)} - \Pi_h u^{(k)}, 
				u_h^{(i)} \right)^2  \\
			    & \leq \left[\rho_h^{(k)}\right]^2 
			    \| u^{(k)} - \Pi_h u^{(k)} \|^2_{L^2(\W)} 
			    \leq C h^{\alpha+\nicefrac12 -\eps} .
			\end{split}
		\end{equation}

		Finally, let us show that 
		\begin{equation}
			\| \w_h^{(k)} - u_h^{(k)} \|_{L^2(\W)} \leq \| \w_h^{(k)} - u^{(k)} \|_{L^2(\W)} ,
			\label{eq:estimacion_proyeccion}
		\end{equation}
		so that
		$$ 
			\| \w_h^{(k)} - u_h^{(k)} \|_{L^2(\W)} \leq \| u^{(k)} - \Pi_h u^{(k)} \|_{L^2(\W)} + \| \Pi_h u^{(k)} - \w_h^{(k)} \|_{L^2(\W)} .
		$$
		Indeed, on one hand we have
		$$
			u_h^{(k)}-\w_h^{(k)} = \left[1 - \left(\Pi_h u^{(k)}, u_h^{(k)}\right)  \right] u_h^{(k)}.
		$$

		On the other hand, due to the normalizations $\|u^{(k)}\|_{L^2(\W)}=\|u_h^{(k)}\|_{L^2(\W)}=1,$ we have 
		$$
			\left| 1 - \|\w_h^{(k)} \|_{L^2(\W)} \right| \leq \|u^{(k)}- \w_h^{(k)} \|_{L^2(\W)}
		$$
		and
		$$
			\|\w_h^{(k)}\|_{L^2(\W)} = \left|\left(\Pi_h u^{(k)}, u_h^{(k)}\right)\right|.
		$$

		Therefore, if we choose the sign of $u_h^{(k)}$ in such a way that $\left(\Pi_h u^{(k)}, u_h^{(k)}\right) \ge 0$, 
		we deduce
		    \begin{align*}
	            \|u_h^{(k)} - \w_h^{(k)} \|_{L^2(\W)} 
	            &= 
	            \left|1 - \left(\Pi_h u^{(k)}, u_h^{(k)}\right)  \right| \\
	            &= 
			\left| 1 -  \left|\left(\Pi_h u^{(k)}, u_h^{(k)}\right)\right| \, \right|\\
			&\le 
			\|u^{(k)} - \w_h^{(k)} \|_{L^2(\W)},
            \end{align*}	 
		as stated in \eqref{eq:estimacion_proyeccion}.

		Thus, estimate \eqref{eq:convL2} is obtained by combining \eqref{eq:triangular}, \eqref{eq:des1}, \eqref{eq:des2} and \eqref{eq:estimacion_proyeccion}.

\end{proof}

\section{Numerical results}
\label{sec:numerico}
This section exhibits the outcome of a variety of experiments carried out by the authors in one- and two-dimensional domains.
Since in general no closed formula for the eigenvalues of the fractional Laplacian is available, we have estimated the order of convergence 
by means of a least-squares fitting of the model 
\[
\lambda^{(k)}_h = \lambda^{(k)} + Ch^\alpha .
\]
This allows to extrapolate approximations of the eigenvalues as well (in the tables we denote this extrapolated value of $\lambda^{(k)}$ as $\lambda_{ext}^{(k)}$).

Throughout this section, the results are compared with those available in the literature. In first place we consider one-dimensional problems, which have been widely studied both theoretically and from the numerical point of view. Next, we show some examples in two-dimensional domains: the unit ball, a square and an $L$-shaped domain.
As for the ball, the deep results of \cite{DydaKuznetsovKwasnicki} allow to obtain sharp estimates on the eigenvalues, and thus provide a point of comparison for the validity of the FE implementation. Regarding the square, some estimates for the eigenvalues are found in \cite{Kwasnicki}. The main interest of the $L$-shaped domain is that, although it does not satisfy the ``standard'' requirements to regularity of eigenfunctions to hold, the numerical order of convergence is the same as in problems posed on smooth, convex domains.
{Finally, Subsection \ref{sub:higher} is concerned with the computation of higher-order eigenspaces.}

General estimates for eigenvalues, valid for a class of domains, are obtained by Chen and Song \cite{ChenSong}. In that paper, the authors state upper and lower bounds for eigenvalues of the fractional Laplacian on domains satisfying the exterior cone condition. Calling $\mu^{(k)}$ the $k$-th eigenvalue of the Laplacian with Dirichlet boundary conditions {on the} domain $\W$, they prove that there exists a constant $C=C(\W)$ such that
\begin{equation} \label{eq:chensong}
C \left(\mu^{(k)}\right)^s \le \lambda^{(k)} \le \left(\mu^{(k)}\right)^s .
\end{equation}
If $\W$ is a bounded convex domain, then $C$ can be taken as $\nicefrac12$. It is noteworthy that, due to the scaling property of the fractional Laplacian, eigenvalues for the dilations of a domain $\W$ are obtained by means of
 $\lambda^{(k)}(\gamma\W) = \gamma^{-2s} \lambda^{(k)}(\W).$

Since we are working with conforming methods, as well as providing approximations, the discrete eigenvalues yield upper bounds for the corresponding continuous eigenvalues. This is of special interest in those cases in which theoretical estimates are not sharp, or the non-symmetry of the domain precludes the possibility of developing arguments such as the ones in \cite{DydaKuznetsovKwasnicki}.

\subsection{One-dimensional intervals} \label{sub:1d}
Eigenvalues for the fractional Laplacian in intervals have been studied by other authors previously. 
In \cite{ZoiaRossoKardar}, a discretized model of the fractional Laplacian is developed, and a numerical study of 
eigenfunctions and eigenvalues is implemented for different boundary conditions. 
In \cite{MR2679702}, the authors deal with one dimensional problems for $s=\nicefrac12$, and provide asymptotic expansion for eigenvalues.
Later, Kwa{\'s}nicki \cite{Kwasnicki} extended this work to the whole range $s\in (0,1)$. Namely, he showed the following identity 
for the $k$-th eigenvalue in the interval $(-1,1)$:
\begin{equation}
\lambda^{(k)} = \left( \frac{k\pi}{2} - \frac{(1-s)\pi}{4}\right)^{2s} + \frac{1-s}{\sqrt{s}} \, O\left(k^{-1}\right) .
\label{eq:kwasnicki}
\end{equation}
Moreover, in that work a method to obtain lower bounds  in arbitrary bounded domains is developed, and it is proved that, on one spacial dimension,
 eigenvalues are simple if $s\ge \nicefrac12$.
As eigenvalues are simple and we are working in one dimension, it is not difficult to numerically estimate the order of convergence of eigenfunctions 
in the $L^2$-norm. Indeed, normalizing the discrete eigenfunctions so that $\| u_h^{(k)} \|_{L^2(-1,1)} = 1$ and choosing their sign adequately, 
these are then compared with a solution on a very fine grid.

On the other hand, in \cite{DuoZhang} it is performed a numerical study of the fractional Schr\"odinger equation in an infinite potential well 
in one spacial dimension. The authors find numerically the ground and first excited states and their corresponding eigenvalues for the stationary linear problem, 
which corresponds to the first two eigenpairs of our equation \eqref{eq:int.autovalores}.

In Table \ref{tab:ordenes1d}, our results for the first $2$ eigenpairs are displayed, {computed over a sequence of uniform meshes with $800$, $1600$, $3200$ and $6400$ elements}. The extrapolated numerical values 
are compared with the estimates from \cite{DuoZhang,Kwasnicki}; the orders of convergence are in good 
agreement with those predicted correspondingly by Theorems \ref{teo:ordenenergia} and \ref{teo:ordenL2}. 
{Moreover, we illustrate the sharpness of Proposition \ref{prop:regintro} by displaying in Figure \ref{fig:intervalo} the first two $L^2$-normalized eigenfunctions for $s=0.1$ and $s=0.9$. As predicted by Remark \ref{remark:bdry}, these functions are smooth within the interval, but behave as $d(x, \pp\W)^s$ near the boundary of the domain.}

\begin{figure}[ht]
	\centering
	\includegraphics[width=0.48\textwidth]{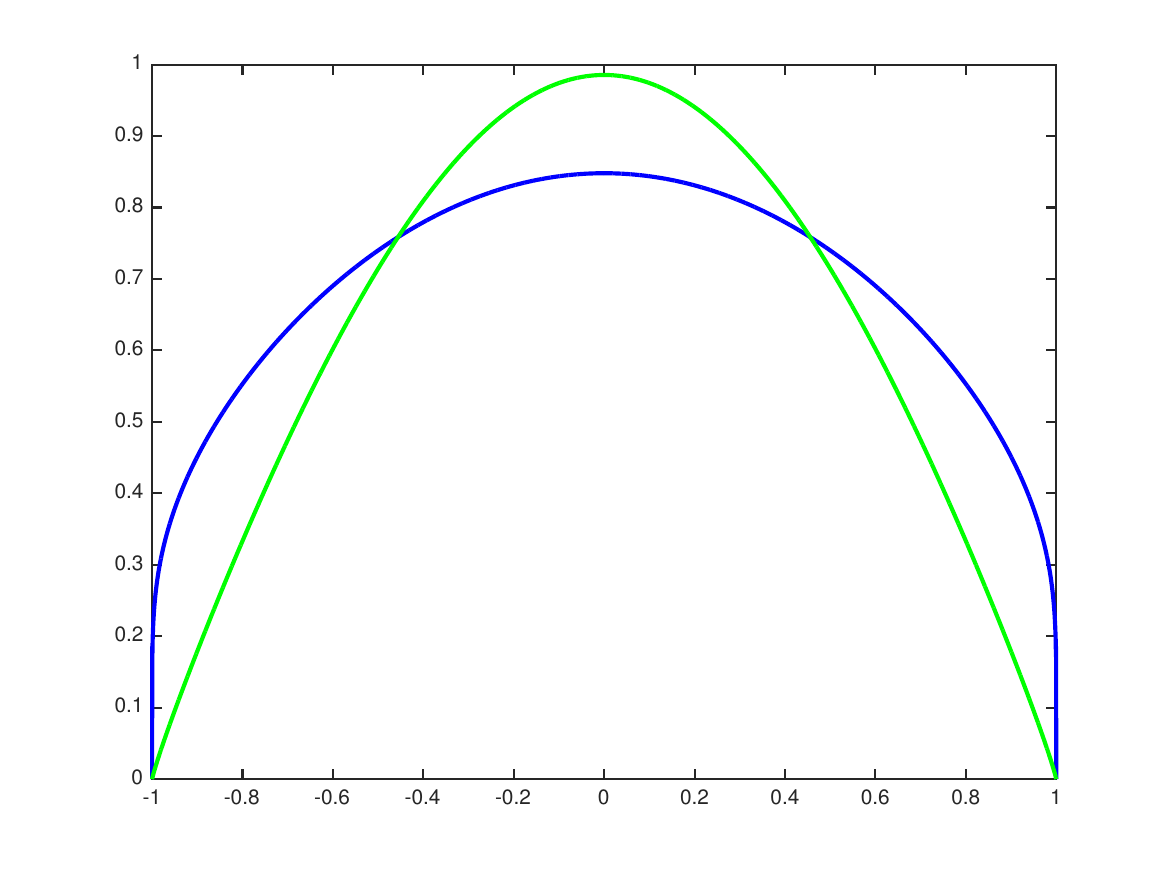} 
	\includegraphics[width=0.48\textwidth]{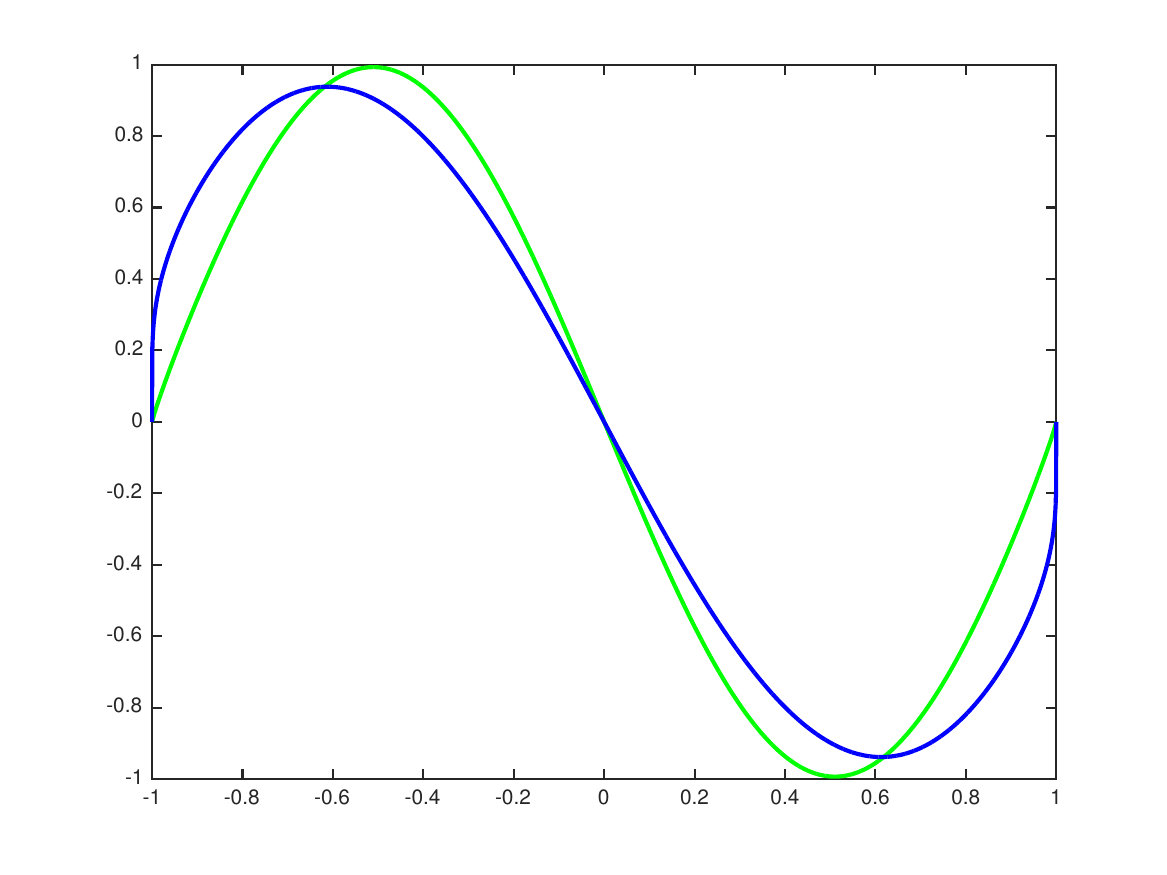} 
	\caption{First (left panel) and second (right panel) $L^2$-normalized discrete eigenfunctions in the interval $(-1,1)$. Color blue corresponds to $s=0.1$, while green corresponds to $s=0.9$.}	\label{fig:intervalo}
\end{figure}

\begin{table}
\caption{First $2$ eigenpairs in the interval $(-1,1)$. On the left, the extrapolated numerical values are compared with the results from \cite{DuoZhang} and with approximation \eqref{eq:kwasnicki}, obtained in \cite{Kwasnicki}. On the right, orders of convergence for eigenvalues and eigenfunctions in the $L^2$-norm (obtained by a least-square fitting) are displayed.}
\label{tab:ordenes1d}
\small{	
\begin{tabular}{ccccccccccc} 
\hline\noalign{\smallskip}
\multicolumn{1}{c}{}	& \multicolumn{6}{c}{Numerical values} & \multicolumn{4}{c}{Orders} \\
\noalign{\smallskip}\hline\noalign{\smallskip}
 $s$   & $\lambda_{ext}^{(1)}$ & $\lambda^{(1)}$ \cite{DuoZhang} & $\lambda^{(1)}$ \cite{Kwasnicki}  & $\lambda_{ext}^{(2)}$ & $\lambda^{(2)}$ \cite{DuoZhang} & $\lambda^{(2)}$ \cite{Kwasnicki}&  $\lambda^{(1)}$ & $\lambda^{(2)}$ & $u^{(1)}$ & $u^{(2)}$ \\ \hline
$0.05$&	$0.9726$								&	$0.9726$								&		$0.9809$								&				$1.0922$								&	$1.0922$								    &	$1.0913$										& $1.108$										  &    	$1.149$	 &     	$0.551$	   &  $0.568$  \\ 
$0.1$&			$0.9575$						& $0.9575$								&		$0.9712$								&			$1.1965$								&	$1.1966$									  &	$1.1948$																		  & $1.071$   		 & $1.102$     & $0.612$  & $0.625$   \\
$0.25$&	$0.9702$								& $0.9702$								&		$0.9908$								&				$1.6015$						&	$1.6016$									  &	$1.5977$										&		$1.021$								  &  $1.038$  		 &  $0.762$   		   & $0.782$   \\  
$0.5$&	$1.1577$								& $1.1578$									&	$1.1781$								&						$2.7548$				&	$2.7549$										&	$2.7488$																				  & $1.001$   		 & $0.979$		  	 & $0.961$  & $0.969$	  \\
$0.75$&	$1.5975$								& $1.5976$								&		$1.6114$								&								$5.0598$		&	$5.0600$									  &	$5.0545$										&			$0.998$							  & $0.999$   		 & $0.998$    		 &  $0.998$  \\ 
$0.9$&	$2.0487$								& ---					 						& $2.0555$									&		$7.5031$									&	---										&				$7.5003$							& $1.004$   		 & $1.021$    		 & $0.999$  & $0.999$   \\
$0.95$&	$2.2481$								& $2.2441$								&		$2.2477$								&		$8.5958$									&	$8.5959$									  &	$8.5942$										&		$1.035$									  & $1.142$   		 &  $0.999$   		   & $0.999$   \\ 
\noalign{\smallskip}\hline
\end{tabular}
}
\end{table}

\subsection{Two-dimensional experiments} \label{sub:2d}
The theoretical order of convergence for eigenvalues is also attained in the following examples in $\R^2$. The implementation for these experiments is based on the code from \cite{ABB}.

\subsubsection*{Unit ball} 
Let us consider the fractional eigenvalue problem on the two-dimensional unit ball. 
In \cite{DKKMeijer}, the weighted operator $u \mapsto (-\Delta)^s (\w^s \, u)$  is studied, where $\w^s (x) = (1-|x|^2)^s_+.$ In particular, explicit formulas for eigenvalues and eigenfunctions of this operator are established. 
Furthermore, in \cite{DydaKuznetsovKwasnicki} the same authors exploit these expressions to obtain two-sided bounds 
for the eigenvalues of the fractional Laplacian in the unit ball in any dimension.
This method provides sharp estimates; however, it depends on the decomposition of the fractional Laplacian as a weighted operator, and the weight $\w^s$ is only explicitly known for the unit ball.

In Table \ref{tab:ordenes_bola}, our results for the first eigenvalue are compared with those of \cite{DydaKuznetsovKwasnicki} for different values of $s$, {computed over a family of uniform meshes with mesh sizes $h \in \{1/30, 1/35, 1/40, 1/45, 1/50 \}$}. This comparison serves as a test for the validity of the code we are employing. As well as the extrapolated value of $\lambda^{(1)}$ and the numerical order of convergence, for every $s$ considered we exhibit an upper bound for the first eigenvalue. These outcomes are consistent with those from \cite{DydaKuznetsovKwasnicki} and the theoretical order of convergence given by Theorem \ref{teo:ordenenergia}.

Computations with graded meshes were carried out for this domain as well. The grading parameter $\mu$ in \eqref{eq:H} was set to be equal to $2$,  {and meshes were taken with about the same total of degrees of freedom as in the experiments with uniform meshes}. For a description of how to build these graded meshes, we refer the reader to Section 5.2 in \cite{AcostaBorthagaray}. In Table \ref{tab:graded_ball} we summarize our findings. For $s \ge\nicefrac12$, we estimated the order of convergence towards the first eigenvalue both with uniform and graded meshes, and also compared the extrapolated value of this eigenvalue. An increase in the convergence rate, in agreement with Theorem \ref{teo:ordenenergia}, is observed.

{
The grading parameter $\mu = 2$ is optimal for every $s$. Indeed, this parameter is in  correspondence with the weight in the regularity estimate from Proposition \ref{prop:regintro}. In this sense, the greater $\mu$ is, the greater the weight can be taken, and thus, the greater the differentiability order of solutions is. Therefore, increasing $\mu$ leads to an increment on the order of convergence with respect to the mesh size parameter. However, if $\mu > 2$ this effect is compensated by the growth in the number of degrees of freedom. We refer the reader to \cite{survey} for details.
}

\begin{table}
\caption{First eigenvalue in the unit ball in  $\R^2$. Estimate from \cite{DydaKuznetsovKwasnicki};  extrapolated value of $\lambda^{(1)}$; upper bound obtained by the finite element method with a meshsize $h\sim 0.02$; numerical order of convergence.} 
\label{tab:ordenes_bola}  
\begin{tabular}{ccccc}
\hline\noalign{\smallskip}
$s$    &  ${\lambda^{(1)}}$& $\lambda_{ext}^{(1)}$ & $\lambda_h^{(1)}$ (UB)	& Order	\\
\noalign{\smallskip}\hline\noalign{\smallskip}
$0.005$ &  $1.00475$	 			& $1.00475$							& $1.00480$	  						& $0.9462$	\\
		$0.05$  &  $1.05095$ 	 			& $1.05094$							& $1.05145$								& $0.9455$	\\
		$0.25$  &  $1.34373$	 			& $1.34367$							& $1.34626$								& $0.9497$	\\
		$0.5$   &  $2.00612$	 			& $2.00607$							& $2.01060$								& $0.9686$	\\
		$0.75$  &  $3.27594$	 			& $3.27632$							& $3.28043$								& $1.0092$	\\
\noalign{\smallskip}\hline
\end{tabular}
\end{table}

\begin{table}
\caption{Computational results in the unit ball in $\R^2$, for uniform and graded meshes. Orders of convergence are stated in terms of the mesh parameter $h$; this behaves like $N^{-1/2}$, $N$ being the number of nodes.} 
\label{tab:graded_ball}
\begin{tabular}{  c  c  c  c  c } \hline
$s $ & Order (unif.) &  Order (graded) & $\lambda^{(1)}_{ext}$ (unif.) & $\lambda^{(1)}_{ext}$ (graded) \\ \hline
$0.5$ & $0.9686$ & $2.1528$ & $2.0061$ & $2.0061$ \\
$0.6$ & $0.9808$ & $2.1720$ & $2.4165$ & $2.4165$ \\
$0.7$ & $0.9969$ & $2.1066$ & $2.9506$ & $2.9507$ \\
$0.8$ & $1.0348$ & $2.0497$ & $3.6494$ & $3.6498$ \\ 
$0.9$ & $1.1654$ & $2.0943$ & $4.5691$ & $4.5695$ \\ \hline
\end{tabular}
\end{table}
\medskip

\subsubsection*{Square} Eigenvalue estimates for the case in which the domain $\W$ is a square in $\R^2$ were also addressed in \cite{Kwasnicki}. 
However, in order to obtain upper bounds, the method proposed in that work depends on 
having pointwise bounds of the Green function for the fractional Laplacian on $\W$. 
The estimates from \cite{ChenSong,Kwasnicki} are compared with our results in Table \ref{tab:ordenes_cuadrado}, where numerical orders of convergence are also displayed. {The computations were carried over a sequence of unstructured uniform meshes with sizes $h \sim \{ 0.1, 0.08, 0.06, 0.04\}$. The upper bound displayed in Table \ref{tab:ordenes_cuadrado} corresponds to the computed result over the finest mesh in this sequence.}

\begin{table}
\caption{First eigenvalue in the square $[-1,1]^2$. Best lower (LB) and upper (UB) bounds known before; upper bound obtained by the finite element method with a meshsize $h\sim 0.04$; extrapolated value of $\lambda^{(1)}$; numerical order of convergence.}
\label{tab:ordenes_cuadrado}
\begin{tabular}{cccccc}
\hline\noalign{\smallskip}
$s$   &  $\lambda^{(1)}$ (LB)& $\lambda^{(1)}$ (UB)  & $\lambda_h^{(1)}$ (UB)& $\lambda_{ext}^{(1)}$ & Order		\\ 
\noalign{\smallskip}\hline\noalign{\smallskip}
$0.05$ &  $1.0308^b$	      	& $1.0831^a$		   			& $1.0412$			   			& $1.0405$			 				& $0.9229$ 	\\
		$0.1$  & 	$1.0506^b$	      	& $1.1731^a$		 			  & $1.0895$			   			& $1.0882$			 				& $0.9230$	\\
		$0.25$ &	$1.1587^b$	      	& $1.4905^a$		  			& $1.2844$			   			& $1.2813$			 				& $0.9283$	\\
		$0.5$  & 	$1.3844^b$	      	& $2.2214^a$		   			& $1.8395$			   			& $1.8344$			 				& $0.9622$	\\
		$0.75$ &	$1.6555^a$	      	& $3.3109^a$		   			& $2.8921$			   			& $2.8872$			 				& $0.9940$	\\
		$0.9$  & 	$2.1034^a$	     		& $4.2067^a$		   			& $3.9492$			   			& $3.9467$			 				& $1.0654$	\\
		$0.95$ & 	$2.2781^a$	     		& $4.5562^a$		   			& $4.4083$			   			& $4.4062$			 				& $1.1496$ \\
\noalign{\smallskip}\hline
\multicolumn{6}{l}{\footnotesize \, $^a$See \cite{ChenSong}. $^b$See \cite{Kwasnicki}.} \\
\end{tabular}
\end{table}

\medskip

\subsubsection*{$L$-shaped domain}
To the authors knowledge, there is no \emph{efficient} method to estimate eigenvalues of the fractional Laplacian if the domain $\W$ lacks symmetry.
The bound \eqref{eq:chensong} remains valid as long as $\W$ satisfies the assumptions required, but the range that estimate provides is quite wide. 

The main advantage of employing the finite element method is that it is flexible enough to cope with a variety of domains.
Moreover, as we are working with conforming approximations, sharp upper bounds for the eigenvalues may be obtained by considering 
discrete solutions on refined meshes.

In Proposition \ref{prop:regintro}, which states that eigenfunctions belong to  $\widetilde{H}^{s+\nicefrac12-\eps}(\W)$, it was assumed that the domain $\W$ satisifies the exterior ball condition. 
For the Laplacian, in order to prove regularity of solutions, it is customary to assume that $\W$ is either smooth or at least convex. In those cases, it is  well known that if $f \in H^{r}(\W)$ for some $r$, then the solutions of the Dirichlet problem with right hand side $f$ belong to $H^{r+2}(\W)$. However, if the domain has a re-entrant corner, solutions are less regular. This also applies to eigenfunctions: in the $L$-shaped domain $\W = [-1,1]^2\setminus [0,1]^2$, the first eigenvalue of the Laplacian is known not to belong to $H^{3/2}(\W)$.

Surprisingly, numerical evidence indicates that eigenvalues of the fractional Laplacian on this $L$-shaped domain converge with the same order as in the previous examples.
This motivates us to conjecture that eigenfunctions and solutions to the Dirichlet equation \eqref{eq:fuente} have the same Sobolev regularity than in smooth domains.

Our findings for the first eigenvalue, {computed over unstructured, uniform meshes with sizes $h \sim \{ 0.1,  0.08, 0.06, 0.04 \}$,} are summarized in Table \ref{tab:ordenes_L}.

\begin{table}
\caption{First eigenvalues in the $L-$shaped domain $[-1,1]^2\setminus [0,1]^2$. 
Upper bound obtained by the finite element method with a meshsize $h\sim 0.04$; extrapolated value of $\lambda^{(1)}$; numerical order of convergence.}
\label{tab:ordenes_L}  
\begin{tabular}{cccc}
\hline\noalign{\smallskip}
$s$  &  $\lambda_h^{(1)}$ (UB)& $\lambda_{ext}^{(1)}$ & Order $\lambda^{(1)}$  \\
\noalign{\smallskip}\hline\noalign{\smallskip}
$0.1$ &  $1.1434$		 					& $1.1413$							& $0.9085$		\\
$0.2$ &  $1.3386$		 					& $1.3342$							& $0.9103$		\\
$0.3$ &  $1.6025$		 					& $1.5956$							& $0.9160$		\\
$0.4$ &  $1.9593$		 					& $1.9499$							& $0.9267$		\\
$0.5$ &  $2.4440$		 					& $2.4322$							& $0.9459$		\\
$0.6$ &  $3.1072$		 					& $3.0936$							& $0.9812$		\\
$0.7$ &  $4.0228$		 					& $4.0069$							& $0.9822$		\\
$0.8$ &  $5.2994$		 					& $5.2831$							& $1.0609$		\\
$0.9$ &  $7.0975$							& $7.0790$							& $1.1891$		\\
\noalign{\smallskip}\hline
\end{tabular}
\end{table}


\subsection{Approximation of high-order eigenspaces}  \label{sub:higher}
In Theorems \ref{teo:ordenenergia} and \ref{teo:ordenL2}, we showed convergence rates of the discrete eigenpairs with respect to the meshsize. The constants involved in the convergence estimates depend on the domain, on $s$, the mesh regularity parameter $\sigma$ (cf. \eqref{eq:regularity.intro}), and importantly, on the eigenvalue number $k$. 

We refer to eigenvalues $\lambda_k$ corresponding to a `large' $k$ as {\em high-order} eigenvalues. We point out that high-order eigenvalues need not to be large in magnitude; it follows from \eqref{eq:chensong} that, for every $k$, $\lim_{s\to 0} \lambda^{(k)}_s = 1$, where $\lambda^{(k)}_s$ denotes the $k$-th eigenvalue of the fractional Laplacian of order $s$. Nevertheless, for a fixed discretization, the quality of the approximation of the $k$-th eigenspace deteriorates as $k$ grows; the discrete system cannot approximate more eigenvalues than the number of degrees of freedom, and the finer-scale oscillations corresponding to high-order eigenvalues cannot be well captured by a coarse mesh. 
A relevant question that arises is, for a fixed  $k$, how many degrees of freedom are needed to provide an approximation of the $k$-th eigenspace within a given tolerance. 

The examples considered in subsections \ref{sub:1d} and \ref{sub:2d} illustrate the order of convergence obtained in theory by examining the first eigenvalues.  Here, we provide numerical examples of the computation of higher-order eigenvalues. In Table \ref{tab:high_interval}, we compute the difference between the finite element approximation of $\lambda^{(100)}_h$ and the asymptotic estimate \eqref{eq:kwasnicki} in the interval $(-1,1)$. We observe that, independently of the value of $s$, the finite element solutions with $8000$ degrees of freedom offer approximations within a relative difference of about $5\times10^{-5}$ with respect to the asymptotic estimate. Since the relative error of the approximation \eqref{eq:kwasnicki} is of the order of $k^{-(1+2s)}$, we deduce that, roughly, the relative error of the finite element approximations of $\lambda^{(100)}$ have a relative error of the order of $10^{-(2+4s)}$ if $s \le 3/4$ and of the order of $10^{-5}$ if $s>3/4$.

\begin{table}
\caption{One-hundredth eigenvalue in the interval $(-1,1)$, computed over a uniform partition of the interval with $8000$ nodes. We compare our approximation with the asymptotic estimate \eqref{eq:kwasnicki}.} 
\label{tab:high_interval}  
\begin{tabular}{cccc}
\hline\noalign{\smallskip}
$s$    &  ${\lambda_h^{(100)}}$& $\lambda^{(100)}$ \cite{Kwasnicki} & Relative difference \\
\noalign{\smallskip}\hline\noalign{\smallskip}
$0.05$ & $1.65735$ & $1.65732$ & $2.3054\times10^{-5}$  \\
$0.1$  &  $2.74690$ & $2.74683$ & $2.6087\times10^{-5}$  \\
$0.25$ & $12.5100$ & $12.5096$ & $3.2075\times10^{-5}$  \\
$0.5$ &  $156.681$ & $156.687$ & $3.9854\times10^{-5}$  \\
$0.75$ & $1965.06$ & $1965.01$ & $2.3987\times10^{-5}$  \\
$0.9$  & $8966.95$ & $8966.54$ & $4.6387\times10^{-5}$  \\
$0.95$ & $14874.9$ & $14873.8$ & $7.4828\times10^{-5}$  \\
\noalign{\smallskip}\hline
\end{tabular}
\end{table}

Computation of high-order eigenvalues is also feasible in more complex geometries. In the experiments carried out over the $L$-shaped domain $ [-1,1]^2\setminus [0,1]^2$, it was observed that the eigenvalue $\lambda_h^{(91)}$ is simple, independently of $s$. 
Figure \ref{fig:L_shaped} (top) displays $L^2$-normalized eigenfunctions for $s = 0.05$  and $s = 0.95$ respectively, computed on a mesh with $11120$ elements ($h \sim 0.04$).
Even though these functions seem to have a similar qualitative behavior, an important difference becomes apparent when inspecting cross sections of these plots (cf. Figure \ref{fig:L_shaped} (bottom)): the eigenfunction corresponding to $s=0.05$ exhibits steeper gradients near the boundary of the domain. The boundary behavior predicted by the regularity theory discussed in Subsection \ref{ss:regularity} extends robustly to high-order eigenvalues.

\begin{figure}[ht]
	\centering
	\includegraphics[width=0.48\textwidth]{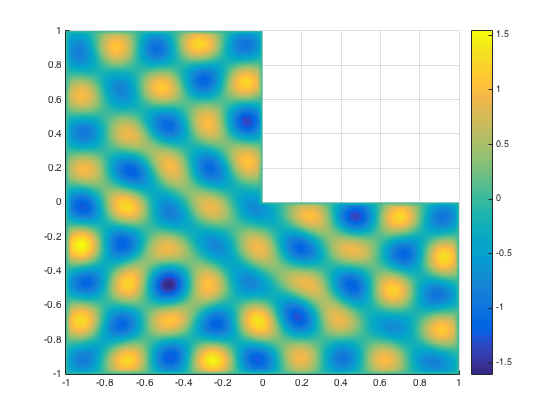}  
	\includegraphics[width=0.48\textwidth]{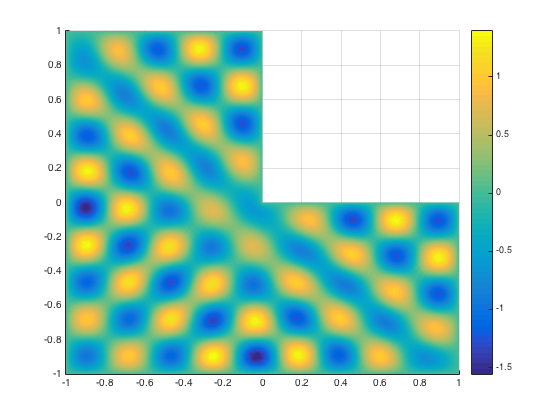} \\
	\includegraphics[width=0.48\textwidth]{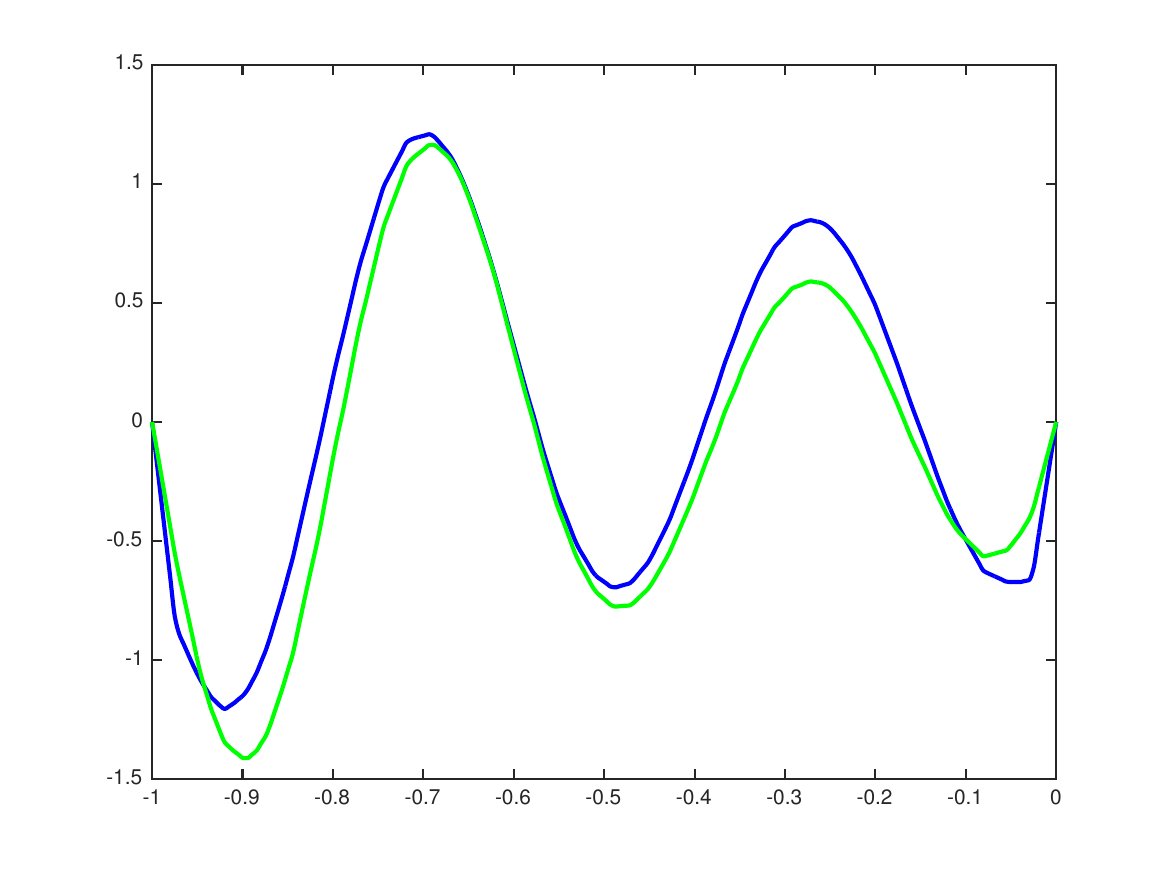} 
	\includegraphics[width=0.48\textwidth]{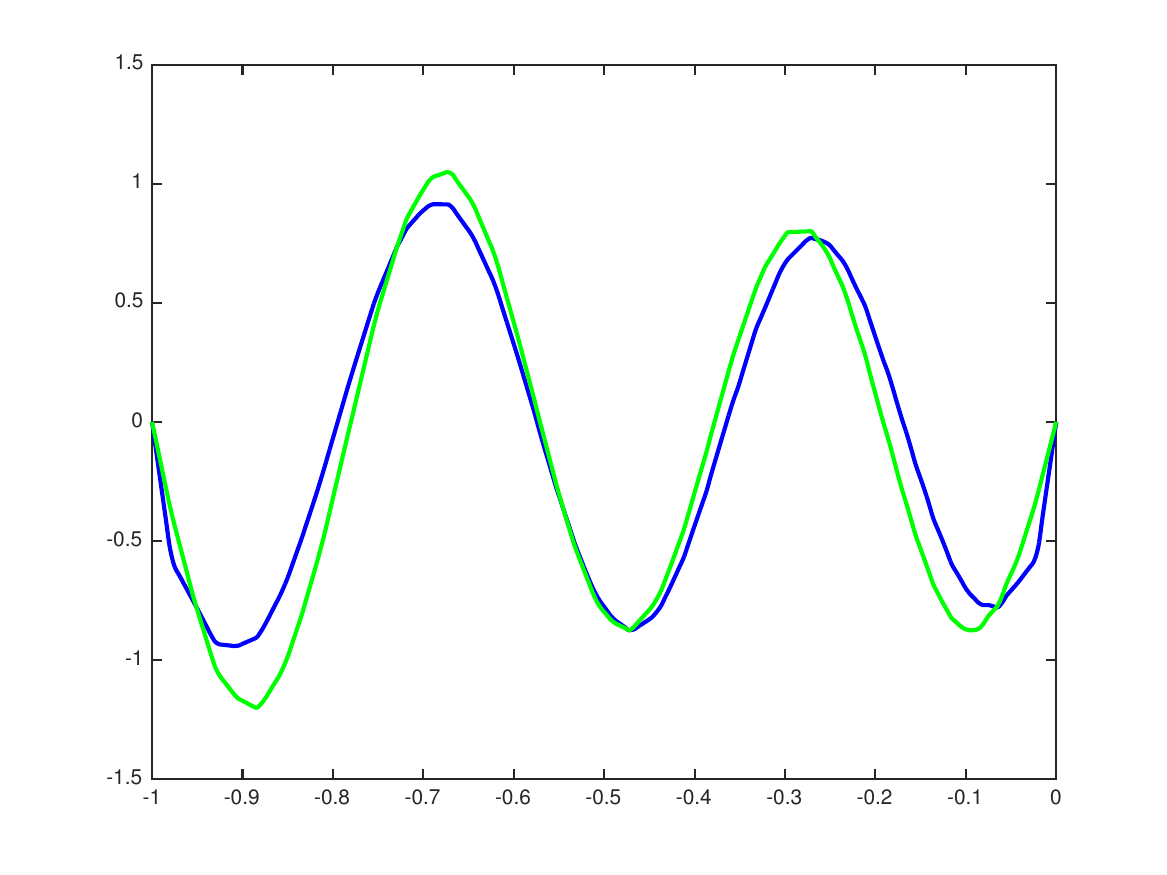} 
	\caption{Finite element approximations of $u^{(91)}$ in the domain $[-1,1]^2\setminus [0,1]^2$. Top: plots of $u_h^{(91)}$ for $s=0.05$ (left) and $s=0.05$ (right), respectively.
Bottom: cross sections at $y=0$ (left) and $y=0.4$; color blue corresponds to $s=0.05$ whereas color green corresponds to $0.95$.}
\label{fig:L_shaped}
\end{figure}

\subsection*{Acknowledgments}
The authors would like to thank G. Grubb and R. Rodr\'iguez for the  valuable help they provided through clarifying discussions on the topic of this paper, and to F. Bersetche for improving the efficiency of the Matlab code employed. 

\bibliographystyle{plain}
\bibliography{bib_autovalores}  

\begin{thebibliography}{10}

\bibitem{ABB}
Gabriel Acosta, Francisco Bersetche, and Juan~Pablo Borthagaray.
\newblock A short {FEM} implementation for a 2d homogeneous {D}irichlet problem
  of a fractional {L}aplacian.
\newblock {\em Comput. Math. Appl.}, 74(4):784--816, 2017.

\bibitem{AcostaBorthagaray}
Gabriel Acosta and Juan~Pablo Borthagaray.
\newblock A fractional {L}aplace equation: Regularity of solutions and finite
  element approximations.
\newblock {\em SIAM Journal on Numerical Analysis}, 55(2):472--495, 2017.

\bibitem{AG}
M.~Ainsworth and C.~Glusa.
\newblock Towards an efficient finite element method for the integral
  fractional {L}aplacian on polygonal domains.
\newblock {\em arXiv:1708.01923v1}, 2017.

\bibitem{Amore}
Paolo Amore, Francisco~M Fern{\'a}ndez, Christoph~P Hofmann, and Ricardo~A
  S{\'a}enz.
\newblock Collocation method for fractional quantum mechanics.
\newblock {\em Journal of Mathematical Physics}, 51(12):122101, 2010.

\bibitem{Antoine}
Xavier Antoine, Qinglin Tang, and Yong Zhang.
\newblock On the ground states and dynamics of space fractional nonlinear
  {S}chr\"odinger/{G}ross--{P}itaevskii equations with rotation term and
  nonlocal nonlinear interactions.
\newblock {\em Journal of Computational Physics}, 325:74--97, 2016.

\bibitem{BO91}
I.~Babu\v{s}ka and J.~Osborn.
\newblock Eigenvalue problems.
\newblock In {\em Handbook of numerical analysis, {V}ol.\ {II}}, Handb. Numer.
  Anal., II, pages 641--787. North-Holland, Amsterdam, 1991.

\bibitem{Bao}
Weizhu Bao, Xinran Ruan, Jie Shen, and Changtao Sheng.
\newblock Fundamental gaps of the fractional {S}chr\"odinger operator.
\newblock {\em arXiv preprint arXiv:1801.06517}, 2018.

\bibitem{BensonWheatcraft}
David~A. Benson, Stephen~W. Wheatcraft, and Mark~M. Meerschaert.
\newblock Application of a fractional advection-dispersion equation.
\newblock {\em Water Resources Research}, 36(6):1403--1412, 2000.

\bibitem{Bertoin}
Jean Bertoin.
\newblock {\em L{\'e}vy processes}, volume 121 of {\em Cambridge Tracts in
  Mathematics}.
\newblock Cambridge University Press, Cambridge, 1996.

\bibitem{Boffi}
Daniele Boffi.
\newblock Finite element approximation of eigenvalue problems.
\newblock {\em Acta Numer.}, 19:1--120, 2010.

\bibitem{Boffia}
Daniele Boffi, Franco Brezzi, and Lucia Gastaldi.
\newblock On the problem of spurious eigenvalues in the approximation of linear
  elliptic problems in mixed form.
\newblock {\em Math. Comp.}, 69(229):121--140, 2000.

\bibitem{survey}
Andrea Bonito, Juan~Pablo Borthagaray, Ricardo~H. Nochetto, Enrique Otarola,
  and Abner~J Salgado.
\newblock Numerical methods for fractional diffusion.
\newblock {\em Computing and Visualization in Science. To appear.}, 2018.

\bibitem{BuadesColl}
A.~Buades, B.~Coll, and J.~M. Morel.
\newblock Image denoising methods. {A} new nonlocal principle.
\newblock {\em SIAM Rev.}, 52(1):113--147, 2010.
\newblock Reprint of ``A review of image denoising algorithms, with a new one''
  [MR2162865].

\bibitem{MR2494809}
Luis Caffarelli and Luis Silvestre.
\newblock Regularity theory for fully nonlinear integro-differential equations.
\newblock {\em Comm. Pure Appl. Math.}, 62(5):597--638, 2009.

\bibitem{CarrHelyette}
Peter Carr, H{\'e}lyette Geman, Dilip~B. Madan, and Marc Yor.
\newblock The fine structure of asset returns: An empirical investigation.
\newblock {\em The Journal of Business}, 75(2):305--332, 2002.

\bibitem{ChenSong}
Zhen-Qing Chen and Renming Song.
\newblock Two-sided eigenvalue estimates for subordinate processes in domains.
\newblock {\em J. Funct. Anal.}, 226(1):90--113, 2005.

\bibitem{clement}
Ph. Cl{\'e}ment.
\newblock Approximation by finite element functions using local regularization.
\newblock {\em Rev. Fran\c caise Automat. Informat. Recherche Op\'erationnelle
  S\'er. RAIRO Analyse Num\'erique}, 9(R-2):77--84, 1975.

\bibitem{RamaTankov}
Rama Cont and Peter Tankov.
\newblock {\em Financial modelling with jump processes}.
\newblock Chapman \& Hall/CRC Financial Mathematics Series. Chapman \&
  Hall/CRC, Boca Raton, FL, 2004.

\bibitem{CushmanGinn}
John~H. Cushman and T.R. Ginn.
\newblock Nonlocal dispersion in media with continuously evolving scales of
  heterogeneity.
\newblock {\em Transport in Porous Media}, 13(1):123--138, 1993.

\bibitem{MR3556755}
Leandro~M. Del~Pezzo and Alexander Quaas.
\newblock Global bifurcation for fractional {$p$}-{L}aplacian and an
  application.
\newblock {\em Z. Anal. Anwend.}, 35(4):411--447, 2016.

\bibitem{MR2895178}
Fran\c{c}oise Demengel and Gilbert Demengel.
\newblock {\em Functional spaces for the theory of elliptic partial
  differential equations}.
\newblock Universitext. Springer, London; EDP Sciences, Les Ulis, 2012.
\newblock Translated from the 2007 French original by Reinie Ern\'e.

\bibitem{Hitchhikers}
Eleonora Di~Nezza, Giampiero Palatucci, and Enrico Valdinoci.
\newblock Hitchhiker's guide to the fractional {S}obolev spaces.
\newblock {\em Bull. Sci. Math.}, 136(5):521--573, 2012.

\bibitem{DuoZhang}
Siwei Duo and Yanzhi Zhang.
\newblock Computing the ground and first excited states of the fractional
  {S}chr\"odinger equation in an infinite potential well.
\newblock {\em Commun. Comput. Phys.}, 18(2):321--350, 2015.

\bibitem{DydaKuznetsovKwasnicki}
B.~Dyda, A.~Kuznetsov, and M.~Kwa{\'{s}}nicki.
\newblock Eigenvalues of the fractional {L}aplace operator in the unit ball.
\newblock {\em J. Lond. Math. Soc.}, 95(2):500--518, 2017.

\bibitem{DKKMeijer}
Bart{\l}omiej Dyda, Alexey Kuznetsov, and Mateusz Kwa{\'{s}}nicki.
\newblock Fractional {L}aplace operator and {M}eijer {G}-function.
\newblock {\em Constructive Approximation}, 45(3):427--448, 2017.

\bibitem{GattoHesthaven}
P.~Gatto and J.S. Hesthaven.
\newblock Numerical approximation of the fractional {L}aplacian via hp-finite
  elements, with an application to image denoising.
\newblock {\em J. Sci. Comp.}, 65(1):249--270, 2015.

\bibitem{Ghelardoni}
Paolo Ghelardoni and Cecilia Magherini.
\newblock A matrix method for fractional {S}turm-{L}iouville problems on
  bounded domain.
\newblock {\em Advances in Computational Mathematics}, 43(6):1377--1401, 2017.

\bibitem{GilboaOsher}
Guy Gilboa and Stanley Osher.
\newblock Nonlocal operators with applications to image processing.
\newblock {\em Multiscale Model. Simul.}, 7(3):1005--1028, 2008.

\bibitem{Grisvard}
P.~Grisvard.
\newblock {\em Elliptic problems in nonsmooth domains}, volume~24 of {\em
  Monographs and Studies in Mathematics}.
\newblock Pitman (Advanced Publishing Program), Boston, MA, 1985.

\bibitem{Grubb}
Gerd Grubb.
\newblock Fractional laplacians on domains, a development of {H}\"ormander's
  theory of $\mu$-transmission pseudodifferential operators.
\newblock {\em Advances in Mathematics}, 268:478 -- 528, 2015.

\bibitem{Grubb_autovalores}
Gerd Grubb.
\newblock Spectral results for mixed problems and fractional elliptic
  operators.
\newblock {\em J. Math. Anal. Appl.}, 421(2):1616--1634, 2015.

\bibitem{Kato}
Tosio Kato.
\newblock {\em Perturbation theory for linear operators}.
\newblock Classics in Mathematics. Springer-Verlag, Berlin, 1995.
\newblock Reprint of the 1980 edition.

\bibitem{Klafter}
Joseph Klafter and Igor~M. Sokolov.
\newblock Anomalous diffusion spreads its wings.
\newblock {\em Physics world}, 18(8):29, 2005.

\bibitem{Kufner}
Alois Kufner.
\newblock {\em Weighted {S}obolev spaces}.
\newblock A Wiley-Interscience Publication. John Wiley \& Sons, Inc., New York,
  1985.
\newblock Translated from the Czech.

\bibitem{MR2679702}
Tadeusz Kulczycki, Mateusz Kwa{\'s}nicki, Jacek Ma{\l}ecki, and Andrzej Stos.
\newblock Spectral properties of the {C}auchy process on half-line and
  interval.
\newblock {\em Proc. Lond. Math. Soc. (3)}, 101(2):589--622, 2010.

\bibitem{Kwasnicki}
Mateusz Kwa{\'s}nicki.
\newblock Eigenvalues of the fractional {L}aplace operator in the interval.
\newblock {\em J. Funct. Anal.}, 262(5):2379--2402, 2012.

\bibitem{Laskin}
Nick Laskin.
\newblock Fractional {S}chr\"odinger equation.
\newblock {\em Physical Review E}, 66(5):056108, 2002.

\bibitem{LionsMagenes}
Jacques~Louis Lions and Enrico Magenes.
\newblock {\em Non-homogeneous boundary value problems and applications},
  volume~1.
\newblock Springer Science \& Business Media, 2012.

\bibitem{YifeiZhang}
Yifei Lou, Xiaoqun Zhang, Stanley Osher, and Andrea Bertozzi.
\newblock Image recovery via nonlocal operators.
\newblock {\em Journal of Scientific Computing}, 42(2):185--197, 2010.

\bibitem{Luchko}
Yuri Luchko.
\newblock Fractional {S}chr\"odinger equation for a particle moving in a
  potential well.
\newblock {\em Journal of Mathematical Physics}, 54(1):012111, 2013.

\bibitem{McCayNarasimhan}
B.~M. McCay and M.~N.~L. Narasimhan.
\newblock Theory of nonlocal electromagnetic fluids.
\newblock {\em Arch. Mech. (Arch. Mech. Stos.)}, 33(3):365--384, 1981.

\bibitem{MetzlerKlafter}
Ralf Metzler and Joseph Klafter.
\newblock The restaurant at the end of the random walk: recent developments in
  the description of anomalous transport by fractional dynamics.
\newblock {\em J. Phys. A}, 37(31):R161--R208, 2004.

\bibitem{RT83}
P.-A. Raviart and J.-M. Thomas.
\newblock {\em Introduction \`a l'analyse num\'erique des \'equations aux
  d\'eriv\'ees partielles}.
\newblock Collection Math\'ematiques Appliqu\'ees pour la Ma\^\i trise.
  [Collection of Applied Mathematics for the Master's Degree]. Masson, Paris,
  1983.

\bibitem{RosOtonSerra}
Xavier Ros-Oton and Joaquim Serra.
\newblock The {D}irichlet problem for the fractional {L}aplacian: {R}egularity
  up to the boundary.
\newblock {\em Journal de Math\'ematiques Pures et Appliqu\'ees}, 101(3):275 --
  302, 2014.

\bibitem{MR3294242}
Xavier Ros-Oton and Joaquim Serra.
\newblock Local integration by parts and {P}ohozaev identities for higher order
  fractional {L}aplacians.
\newblock {\em Discrete Contin. Dyn. Syst.}, 35(5):2131--2150, 2015.

\bibitem{ScottZhang}
L.~Ridgway Scott and Shangyou Zhang.
\newblock Finite element interpolation of nonsmooth functions satisfying
  boundary conditions.
\newblock {\em Math. Comp.}, 54(190):483--493, 1990.

\bibitem{MR3089742}
Raffaella Servadei.
\newblock The {Y}amabe equation in a non-local setting.
\newblock {\em Adv. Nonlinear Anal.}, 2(3):235--270, 2013.

\bibitem{Servadei20132445}
Raffaella Servadei and Enrico Valdinoci.
\newblock A brezis-nirenberg result for non-local critical equations in low
  dimension.
\newblock {\em Communications on Pure and Applied Analysis}, 12(6):2445--2464,
  2013.

\bibitem{MR3002745}
Raffaella Servadei and Enrico Valdinoci.
\newblock Variational methods for non-local operators of elliptic type.
\newblock {\em Discrete Contin. Dyn. Syst.}, 33(5):2105--2137, 2013.

\bibitem{ServadeiValdinoci}
Raffaella Servadei and Enrico Valdinoci.
\newblock On the spectrum of two different fractional operators.
\newblock {\em Proc. Roy. Soc. Edinburgh Sect. A}, 144(4):831--855, 2014.

\bibitem{MR3161511}
Raffaella Servadei and Enrico Valdinoci.
\newblock Weak and viscosity solutions of the fractional {L}aplace equation.
\newblock {\em Publ. Mat.}, 58(1):133--154, 2014.

\bibitem{MR3271254}
Raffaella Servadei and Enrico Valdinoci.
\newblock The {B}rezis-{N}irenberg result for the fractional {L}aplacian.
\newblock {\em Trans. Amer. Math. Soc.}, 367(1):67--102, 2015.

\bibitem{Silling}
S.~A. Silling.
\newblock Reformulation of elasticity theory for discontinuities and long-range
  forces.
\newblock {\em J. Mech. Phys. Solids}, 48(1):175--209, 2000.

\bibitem{MR2270163}
Luis Silvestre.
\newblock Regularity of the obstacle problem for a fractional power of the
  {L}aplace operator.
\newblock {\em Comm. Pure Appl. Math.}, 60(1):67--112, 2007.

\bibitem{Valdinoci}
Enrico Valdinoci.
\newblock From the long jump random walk to the fractional {L}aplacian.
\newblock {\em Bol. Soc. Esp. Mat. Apl. S$\vec{\rm e}$MA}, 49:33--44, 2009.

\bibitem{zhou2015fractional}
Yuan Zhou.
\newblock Fractional {S}obolev extension and imbedding.
\newblock {\em Transactions of the American Mathematical Society},
  367(2):959--979, 2015.

\bibitem{ZoiaRossoKardar}
A.~Zoia, A.~Rosso, and M.~Kardar.
\newblock Fractional {L}aplacian in bounded domains.
\newblock {\em Phys. Rev. E (3)}, 76(2):021116, 11, 2007.

\end{thebibliography}

\end{document}